\documentclass[a4paper]{scrarticle}
\usepackage{a4wide}
\usepackage{amscd,amssymb,amsmath,amsthm}
\usepackage{graphicx}
\usepackage[backend=biber, style=ieee,giveninits=true,sorting=anyt]{biblatex}
\usepackage{bbm}
\usepackage[autostyle]{csquotes}
\usepackage{hyperref}
\usepackage{doi}
\hypersetup{
    colorlinks=true,
    linkcolor=blue,
    filecolor=black,      
    urlcolor=blue,
    citecolor=blue,
}
 \usepackage{todonotes}

\newcommand{\invisible}[1]{}
\usepackage[section]{placeins}
\usepackage[shortlabels]{enumitem}
\usepackage[lofdepth,lotdepth]{subfig}
\usepackage{tikz}
\usepackage{todonotes}
\usetikzlibrary{arrows,shapes,decorations,automata,backgrounds,petri,bending,trees,decorations.pathreplacing,calc}

\newcommand{\R}{\mathbb{R}}
\newcommand{\N}{\mathbb{N}}
\newtheorem{thm}{Theorem}[section]
\newtheorem{defn}[thm]{Definition}
\newtheorem{lemma}[thm]{Lemma}
\newtheorem{pro}[thm]{Proposition}
\newtheorem{rk}[thm]{Remark}
\newtheorem{cor}[thm]{Corollary}

\newtheorem{con}[thm]{Conjecture}
\newtheorem{ex}[thm]{Example}

\newcommand{\footnoteremember}[2]{
	\footnote{#2}
	\newcounter{#1}
	\setcounter{#1}{\value{footnote}}
}
\newcommand{\footnoterecall}[1]{
	\footnotemark[\value{#1}]
}
\newcommand{\sss}{\scriptscriptstyle}

\newcommand{\dphigeqalphadir}{d_\phi^{\sss(-,\geq \alpha)}}
\newcommand{\nn}{\nonumber}
\DeclareSymbolFont{extraup}{U}{zavm}{m}{n}
\DeclareMathSymbol{\varheart}{\mathalpha}{extraup}{86}
\DeclareMathSymbol{\vardiamond}{\mathalpha}{extraup}{87}

\newcommand{\conn}{\longrightarrow}

\begin{document}
\title{Power-law hypothesis and (un)fairness of PageRank on undirected multi-type PAMs}
\author{
Christian Borgs \footnoteremember{Ber}{EECS department, University of California, Berkeley; 
borgs@berkeley.edu} 
\and Florian Henning\footnoteremember{TUe}{Eindhoven University of Technology, Department of Mathematics \& Computer Science;\\ 
f.b.henning@tue.nl, r.w.v.d.hofstad@tue.nl, n.v.litvak@tue.nl}
		\and  Remco van der Hofstad \footnoterecall{TUe}
  \and
	Nelly Litvak\footnoterecall{TUe}
}
\date{\today} 	
\maketitle

\begin{abstract}
The preferential attachment model (PAM) describes the sequential growth of a network based on the ``rich-get-richer" principle. Several versions of it have become established for modeling, e.g., citation networks, capturing a power-law degree distribution. \\
Directed versions of the preferential attachment model where the edges are directed from the new to the old vertices have been the subject of extensive research. They have been shown to exhibit remarkable properties such as heavier tails for the limiting graph-normalized PageRank than for the in-degrees. By contrast, for the undirected version, we recently showed that PageRank has similar tails as the degree.
In the present paper, we discuss the PageRank asymptotics for a multi-type version of the undirected PAM (here vertices have different colors), complementing previous results of Antunes, Bhamidi, Banerjee and Pipiras on the asymptotics of PageRank on similar directed multi-type or colored PAMs.
Our studies are motivated by the aim to go beyond the rigid rule of edge orientation in directed preferential attachment models.

As the main result, for the case of a finite set of colors, we show that the power-law hypothesis for PageRank is fulfilled also for the colored undirected PAM, where, by contrast to the directed case, the power-law exponent is color-dependent for some choices of the initial color distribution and the attractiveness function.
For the specific case of a two-type model, we discuss implications of our results on fairness in sampling underrepresented nodes from the network.

\end{abstract}

\noindent
\textbf{Mathematics Subject Classifications (2020).} 05C80 (primary);\\
60J80, 60B20 (secondary)\\
\textbf{Key words.} PageRank, scale-free network, heavy tails, local weak convergence, color.
\section{Introduction}
The preferential attachment model (PAM) was initially introduced by Barab\'asi and Albert in \cite{BaAl99} to model the structure and growth of the web graph and its particular property of being scale free.
In PAMs, at each time step, a new vertex is added to the existing graph. The core idea is that this new vertex has $m$ half-edges to connect to the already existing vertices in the network, where high-degree vertices are more likely to be connected to (\textit{preferential attachment}). In the initial formulation in \cite{BaAl99}, the probability of the new vertex in the network to connect an edge to a specific vertex in the graph is \textbf{proportional} to the degree of that vertex (up to normalization), with the additional constraint of simplicity, which means that no two of the $m$ half edges are allowed to connect to the same existing vertex in the network.  Since \cite{BaAl99}, several variants and generalizations of the original Barab\'asi-Albert model have been studied, many of them describing multi-graphs instead of simple graphs (see also \cite[Ch.\ 8]{vdH17} for a literature overview on common variations of the PAM).  

In this paper, we are mainly interested in the case of undirected PAM with a finite number of vertex colors.
We start with a short summary of some of the landmarks in studying versions of the PAM that will become important for the present paper. 

For the case of the undirected PAM with an arbitrary finite number of colors, Jordan \cite[Theorem 2.2]{Jo13} provides the precise power-law asymptotics of the degree sequence. The proof is based on studying the discrete dynamical system of the proportion of edge ends that connect to a vertex with a specified color. Furthermore, in the case of colors in an arbitrary (possibly infinite) space (see \cite{FlAbFrVe06} for an introduction to that case), and building upon the formerly mentioned result, \cite[Theorem 2.3]{Jo13} gives power-law bounds on the limiting degree-sequence. 
These result extend earlier results on the PAM with \textit{fitness} \cite{BoChDaRo07}, where the linear attachment rule is perturbed by a multiplicatively acting random attribute (\textit{fitness}) of the target vertex. 
Going beyond the degree sequence, Berger et al. \cite{BeBoCh14} first provided a full description of the \textbf{local limit} of the undirected preferential attachment model without colors. Here, the proof in \cite{BeBoCh14} is based on exploiting the close relation between the construction of the PAM and the update rule in a P\'olya urn model by means of a \textit{P\'olya urn representation} of the PAM.\\
In short, \textit{local} (weak) convergence of sequences of (random or deterministic) \textbf{undirected} graphs (cf. \cite{AldSte04, BeSc01}) is a notion defined on the space of connected rooted graphs up to root-preserving isomorphisms, into which any graph can be embedded by choosing a vertex uniformly at random from the graph as a root, and restricting to the connected component around that vertex. 
Note that the limiting behavior of (root-)PageRank is uniquely determined by the local limit (if it exists) \cite{GavdHLi20,vH24,vHPa25}.
For sequences of \textbf{directed} (random or deterministic) graphs, similar notions of local convergence exist, where the notion used in  \cite{GavdHLi20}, which will be particularly relevant for the remainder of this paper, is based on the topology of \textit{vertex-marked} incoming neighborhoods, with the vertex marks denoting the vertex out-degrees. Following \cite{GavdHLi20}, any neighborhood of a vertex in the directed graph contains only vertices that are \textit{explorable} by backtracking paths of \textbf{incoming} edges.
In the case of the PAM without colors, the known weak limit for the undirected case is sufficient to directly deduce the corresponding directed local limit (in the above sense) from it.
This observation was first used in \cite{BaOC22} in the context of the PageRank asymptotics for the \textbf{directed} PAM without marks, where the local limit for $m \geq 2$ was taken from the undirected case discussed in \cite{BeBoCh14} and then extended by a new representation in terms of a stopped continuous-time branching process.
For the directed version of the  preferential attachment model with marks in an arbitrary finite space, \cite[Theorem 4.3]{AnBaBhPi26} provides a full description of the local limit of in-components. Interestingly, in contrast to \cite{BeBoCh14}, the proof of \cite[Theorem 4.3]{AnBaBhPi26} is not based on a P\'olya urn representation, but on a generalization of the tree case ($m=1$), which is proved in \cite[Theorem 3.5]{AnBaBhPi26} employing the notion of \textit{fringe convergence}. Fringe convergence is a tree-specific formulation of local convergence introduced by Aldous in \cite{Al91}. This allows the authors of \cite{AnBaBhPi26} to prove directed local convergence by means of an adaption of the dynamical system approach of \cite{Jo13}. However, making heavy use on the orientation of edges, the proof of \cite{AnBaBhPi26} does not have an obvious adaption to the undirected set-up.    
For the case of the \textit{generalized} undirected preferential attachment model without colors, but with a random offspring distribution for the new vertices, \cite{DevEvH09} provides a description of the limiting behavior of the degree sequence, while, under a moment assumption on the offspring distribution of the new vertices, \cite{GaHaHoRa22} extends this result to a full description of the local limit of the graph sequence.

\paragraph{Main contribution of this paper.}
We prove power-law bounds for the PageRank at a uniformly chosen vertex from an undirected multi-type PAM conditioned on its color (see Theorem \ref{thm: Power-law bounds} below). In contrast to the corresponding known result for the directed model with edges directed from younger to older vertices, the power-law exponents for the PageRank in the undirected model can exhibit a dependence on the type of the vertex (see Corollary \ref{cor: exponents} below).    
As an application to fairness in sampling nodes from an undirected two-type network, where vertices of one type form an underrepresented (regarding the proportion of vertices in the network) minority, we compare sampling by PageRank with uniform sampling and sampling by degree for large networks and small values of the damping factor $c$ (see Proposition \ref{pro: MinorityFraction} below). In the case of a symmetric attractiveness kernel, we illustrate our results by simulations.

\subsection{Outline of the paper}
In Section \ref{subsec: setup} we introduce the notational framework that is necessary for the remainder of the paper. Section \ref{sec: Result} then presents the main new results: color-dependent power-law bounds for the PageRank distribution (Theorem \ref{thm: Power-law bounds}) and an expression for the corresponding power-law exponents (Corollary \ref{cor: exponents}). In Section \ref{Sec: Applications} we take a closer look at the case of a two-element color space and shed light on the implications of our results in view of fairness in sampling nodes from the networks. Section \ref{Sec: Applications} contains both abstract theoretical results, as well as illustrative plots of the relevant quantities contrasted with a simulation study.
Afterwards, the proofs are given in Section \ref{sec: proofs}.
First, in Section \ref{subsec: upper bound} we provide the proof of the upper power-law bound of Theorem \ref{thm: Power-law bounds}, which is a consequence of the degree asymptotic established in \cite[Theorems 2.1 and 2.2]{Jo13}, together with the result of \cite[Theorem 1.1]{HevHLi25a} that states that in undirected networks the PageRank of any vertex is always bounded from above by the respective degree. In contrast to that, the subsequent proof of the lower bound outlined in Section \ref{subsec: lower bound} is more involved and makes use of the fact that a fixed out-degree of new vertices allows us to deduce a suitable power-law bound on PageRank in the \textbf{undirected} case from the known (\cite[Theorem 4.3]{AnBaBhPi26}) local a.s.-limit of the \textbf{directed} version of the multi-type PAM. 
Section \ref{sec: proofs} then closes with the proof of Proposition \ref{pro: MinorityFraction} 
on the role of the PageRank-shares of vertices of different types for fairness in sampling representative nodes from the network. 
Finally, the paper finishes with Section \ref{sec: discussion}, which provides a discussion of the scope of the paper and an outlook on related research questions.  
\subsection{PageRank on undirected PAM with finite set of colors}\label{subsec: setup}
In this subsection we establish the basic notions that we will employ and the definition of the model that is central to this paper.

\paragraph{Undirected PAM with a finite set of colors.}
We adapt the notations in \cite{Jo13,AnBaBhPi26,HevHLi25a} to establish our formal framework.
Let $S=\{1,2,\ldots,q\}=:[q]$, where $q \in \N$, denote the finite set of colors, and let $\mu$ be any strictly positive fixed probability measure on $(S,2^S)$, where we abbreviate $\mu_i:=\mu(\{i\})$ for the atoms.  
Further, let $\kappa\colon S \times S \rightarrow (0,\infty)$ denote the \textit{attractiveness} function. Finally, let $\boldsymbol{m}=(m_1,\ldots,m_q)^T \in \mathbb{N}^q$ be a fixed deterministic vector, where as in the reminder of this paper, boldface fonts denote vectors or matrices.

We define a sequence $(G_n)_{n \in (\N \cup \{0\})}$ of undirected colored multi-graphs without self-loops that grows sequentially as follows:\\
Initially, at time $n=0$, there is a base graph $G_0$ with $n_0$ vertices, where each vertex $v$ in $G_0$ has a color $x_v \in S$. 
For any time $n+1>0$, given the sequence $(G_s)_{s \leq n}$, let $v_{n+1}$ denote the new vertex that is added to the graph $G_n$.
First, the vertex $v_{n+1}$ is assigned a color $X_{v_{n+1}}$ according to $\mu$ independently of everything else.
Second, conditionally on $X_{v_{n+1}}=k$, the vertex $v_{n+1}$ is assigned $m_k$ half-edges, which we denote by $v^{\sss(1)}_{n+1},\ldots,v^{\sss(m_k)}_{n+1}$. 
Third, conditionally on $G_n$ and $X_{v_{n+1}}=k$, the $m_{k}$ half-edges of vertex $n+1$ are independently attached to vertices in the existing graph $G_n$ according to
\begin{equation}\label{eq: AttachmentRule}
\mathbb{P}(v^{\sss(h)}_{n+1}\conn v \mid G_n,\ X_{v_{n+1}}=k)=\frac{\kappa(X_v,k)d_{v}(G_n)}{\sum_{v' \in G_n}\kappa(X_{v'},k)d_{v'}(G_n)}, 
\end{equation}
where $d_{v'}(G_n)$ is the degree of vertex $v$ in $G_n$ \textbf{before} any of the half-edges of $v_{n+1}$ has been attached to the graph. 
Note that due to the specific form of the attachment rule \eqref{eq: AttachmentRule}, the model is invariant under multiplication of $\kappa$ by a positive constant. 
\begin{rk}[Attachment rule in non-tree case: independent vs.\ sequential model]
The attachment rule \eqref{eq: AttachmentRule}, which is the basis of all results taken from \cite{Jo13} and \cite{AnBaBhPi26}, does not consider an intermediate update of the degrees in the course of subsequentially attaching the half-edges of the new vertex. In the terminology of \cite{BeBoCh14}, this describes the \textbf{independent model}, where the results of \cite{BeBoCh14} regarding the undirected PAM without colors refer to the \textbf{sequential model}, in which the degrees occurring in the attachment rule are updated after each half-edge attachment. 
\end{rk}
\noindent
\paragraph{Graph-normalized PageRank.}
For any undirected finite graph $G=(V(G),E(G))$ with vertex set $V(G)$ and edge set $E(G)$, the \textit{graph-normalized} (normalized by the number of vertices) \textit{PageRank} $(R_v)_{v \in V}$ with damping factor $c \in (0,1)$ is defined as the unique solution to 
\begin{equation}
R_v=R_v(G)=c\sum_{\substack{w \in V \\w \sim v}} \frac{1}{d_w}R_w+1-c,    
\end{equation}
where we write $w \sim v$ when $v$ and $w$ are vertices in $G$ that are connected by an edge, i.e., the sum runs over all nearest neighbors of $v$. In case of multi-edges, the respective term is multiplied by the multiplicity of the edge.

\section{Result: Color-dependent power-law for root PageRank}\label{sec: Result}
In order to state our first main result, we introduce some definitions and objects from \cite{AnBaBhPi26}. The meaning and intuition behind these will be explained later on. Note that our notation partially differs from that in \cite{AnBaBhPi26}.\\
Let $\boldsymbol{\eta}^{\boldsymbol{m}}=(\eta^{\boldsymbol{m}}_1,\ldots,\eta^{\boldsymbol{m}}_q)$ be the unique minimizer for the
convex function  
\begin{equation}\label{eq: LyapunovFunction}
V^{\boldsymbol{m}}(y):=\sum_{i=1}^q y_i-\frac{1}{2}\sum_{k=1}^q m_k\mu_k \left[\log(y_k)+\log \left(\sum_{l=1}^q y_l\kappa(l,k)\right)\right], \quad \boldsymbol{y} \in [0,\infty)^q,\end{equation}
in the interior of the set
\[\Big\{\boldsymbol{y} \in [0,\infty)^q \mid \sum_{k=1}^qy_k=\sum_{k=1}^q\mu_k m_k \Big\}.\]
The function $V^{\boldsymbol{m}}$ first appeared in \cite{Jo13} for the case of $m_k=m_1$ for all $k \in S$, i.e., when the number of half-edges of a new vertex does not depend on its color. 
For $k,l \in S$, set \begin{equation}\label{eq: varphik}
\varphi_{k,l}^{\boldsymbol{m}}:=\frac{m_l\kappa(k,l)\mu_l}{\sum_{l' \in S}\kappa(l',l)\eta^{\boldsymbol{m}}_{l'}}
\quad \text{ and } \quad  \varphi_k^{\boldsymbol{m}}:=\sum_{l \in S}\varphi_{k,l}^{\boldsymbol{m}}.\end{equation}
For our next theorem, we only need $\varphi_k^{\boldsymbol{m}}$, but the definition of $\varphi_{k,l}^{\boldsymbol{m}}$ is relevant in what follows.
Finally, for functions $f,g\colon \mathbb{R} \rightarrow \mathbb{R}$, we write, for $a \in \mathbb{R} \cup \{-\infty,+\infty \}$, that $f = o(g)$ as $x$ tends to $a$ if and only if $\limsup_{x \rightarrow a}\vert \frac{f(x)}{g(x)} \vert=0$ and $f = O(g)$ as $x$ tends to $a$ if and only if $\limsup_{x \rightarrow a} \vert\frac{f(x)}{g(x)} \vert <\infty$.
With these definitions in hand, our main result is as follows:
\begin{thm}[Power-law bounds for PageRank in undirected multi-type PAM]\label{thm: Power-law bounds}
Consider the undirected multi-type PAM $(G_n)_{n \in \N}$ with finite set of colors $S=\{1,2,\ldots,q\}$ with $q \in \N$. Further, for any $n \in \N$, let $\phi_n$ denote a vertex that is chosen uniformly at random from $G_n$, and let $c \in (0,1)$.
Then there exists a constant $\gamma \in (1,\infty)$ that does not depend on the color $k$ such that, as $r \rightarrow \infty$, the graph-normalized PageRank with damping factor $c$ satisfies
\begin{align}
&\liminf_{n \rightarrow \infty} \mathbb{P}(R_{\phi_n}(G_n)>r \mid X_{\phi_n}=k) \geq (1+o(1))\frac{2\Gamma(m_k+\frac{2}{\varphi_k^{\boldsymbol{m}}})}{\varphi_k^{\boldsymbol{m}} \Gamma(m_k)}\sum_{s=\lceil \gamma r \rceil}^\infty\frac{\Gamma(s)}{\Gamma(s+\frac{2} {\varphi_k^{\boldsymbol{m}}}+1)} \cr
&\limsup_{n \rightarrow \infty} \mathbb{P}(R_{\phi_n}(G_n)>r \mid X_{\phi_n}=k) \leq \frac{2\Gamma(m_k+\frac{2}{\varphi_k^{\boldsymbol{m}}})}{\varphi_k^{\boldsymbol{m}} \Gamma(m_k)}\sum_{s=\lceil r \rceil}^\infty\frac{\Gamma(s)}{\Gamma(s+\frac{2} {\varphi_k^{\boldsymbol{m}}}+1)}.  
\end{align}

\end{thm}
We will prove the upper and the lower power-law bound separately in Propositions \ref{pro: Jordan} (upper bound) and \ref{pro: LowerBound} (lower bound) below. 
By Stirling's formula, as $s$ tends to infinity, 
\begin{equation*}
 \frac{\Gamma(s)}{\Gamma(s+\frac{2}{\varphi_k^{\boldsymbol{m}}}+1)}=\left(1+O(s^{-1})\right)s^{-\lambda_k^{\boldsymbol{m}}},
\end{equation*}
where $\lambda_k^{\boldsymbol{m}}:=1+\frac{2}{\varphi_k^{\boldsymbol{m}}}.$
In particular, by Theorem \ref{thm: Power-law bounds} and the Euler-Maclaurin formula, as $r$ tends to infinity, we obtain the following power-law bounds for PageRank:

\begin{cor}[Exponents in power-law bounds]\label{cor: exponents}
For $k \in S$,
\begin{align}
&\liminf_{n \rightarrow \infty} \mathbb{P}(R_{\phi_n}(G_n)>r \mid X_{\phi_n}=k) \geq (1+o(1))\frac{\Gamma(m_k+\frac{2}{\varphi_k^{\boldsymbol{m}}})}{\Gamma(m_k)}\gamma^{-(\lambda_k^{\boldsymbol{m}}-1)}r^{-(\lambda_k^{\boldsymbol{m}}-1)}\cr
&\limsup_{n \rightarrow \infty} \mathbb{P}(R_{\phi_n}(G_n)>r \mid X_{\phi_n}=k) \leq (1+o(1))\frac{\Gamma(m_k+\frac{2}{\varphi_k^{\boldsymbol{m}}})}{\Gamma(m_k)}r^{-(\lambda_k^{\boldsymbol{m}}-1)}.  
\end{align}
Thus, the exponent for the limiting power-law tails of the PageRank of a uniformly chosen vertex of color $k$ is $\lambda_k^{\boldsymbol{m}}-1=\frac{2}{\varphi_k^{\boldsymbol{m}}}$.
\end{cor}
\section{Application to fairness in network sampling and simulations}\label{Sec: Applications}
In this section, we discuss the specific case of the PAM with a two-element set of colors. Structurally, this is the simplest non-trivial example of a multi-type PAM. From an application-oriented point of view, this set-up is relevant for modeling (un)fairness in social networks (see \cite{KaGeWaSiSt18}, the related \textit{biased PAM} of \cite{AvKeLoMaPePi15}, and \cite{StLiCh24}). If we choose $\mu_1<\mu_2$, then vertices of color 1 (displayed in red in the plots below) become a minority in the network, while vertices of color 2 (displayed in blue) form the majority. As an extension to the discussion in \cite[Chapter 6]{AnBaBhPi26} for the directed (tree) setting, in this section we will compare uniform sampling of vertices from the network to sampling by PageRank. More precisely, we will compare the proportion of minority vertices to the proportion of the total PageRank of all minority vertices among the total PageRank of all vertices in the network.\\ 
The outline is as follows. First in Section \ref{sec: example setup}, we introduce the relevant notational framework, and provide a three-dimensional plot of the parameter-dependent power-law exponents of the degrees and PageRanks. Afterwards, in Section \ref{sec: small-c approximation}, we present a second order expansion (in the damping factor $c$) of the proportion of the total PageRank of all minority vertices among the total PageRank of all vertices, and give explicit asymptotics for the linear term (see Proposition \ref{pro: MinorityFraction}). Based on numerical evaluations of these asymptotics presented in Figure \ref{fig: PageRankShares}, we finally formulate and discuss Conjecture \ref{conj: minority pr}.

\subsection{Different power-law exponents for PageRank in model with two colors} 
 \label{sec: example setup}
The following example has already been discussed in \cite[Corollary 4.5]{AnBaBhPi26} for the directed multi-type PAM, and for the particular case of color-independent out-degrees also in \cite[Section 2.1.3]{Jo13}. In addition to the discussions in \cite{AnBaBhPi26} and \cite{Jo13}, in Figure \ref{fig:lambda12} below, we provide a three-dimensional plot of the relevant power-law exponents in Corollary \ref{cor: exponents} as a function of the model parameters, while in Figure \ref{fig: SimulationOFTwoTypePAM}, we compare the results of an actual simulation of the network for specific choices of the model parameters to theoretical values.
 
 \begin{ex}[{\protect Consequence of \cite[Corollary 4.5]{AnBaBhPi26} and \cite[Section 2.1.3]{Jo13}}]\label{cor: parameters two-colored}
 Let $S=\{1,2\}$ consist of two colors and assume that the attractiveness function $\kappa$ is symmetric in that there is some $\varepsilon \in (0,1)$ such that
\begin{equation}\label{eq: symmetricKappa}
\kappa(1,2)=\kappa(2,1)=\varepsilon, \quad 
 \kappa(1,1)=\kappa(2,2)=1-\varepsilon.
\end{equation}
 Further we write $a:=\mu_1$, i.e., $\mu_2=1-a$.
 Then the power-law exponents $\lambda_1$ and $\lambda_2$ for the limiting degree distribution (and hence, also for the limiting root-PageRank) differ if and only if
  \[m_1a(M-\eta_1^{m})<m_2(1-a)\eta_1^{m},\]
  where $M:=am_1+(1-a)m_2$ denotes the expected out-degree
 In the particular case $m_1=m_2=m \in \N$, let $y=\tilde{\eta}_1$ denote the unique root in $(0,1)$ of the $m$-independent equation
 \begin{equation}\label{eq: example derivative}
\begin{split}
0=&2y^3(1-4\varepsilon(1-\varepsilon))+y^2(-2(1+a)+9\varepsilon+6a\varepsilon-10\varepsilon^2-4a\varepsilon^2)\cr
&\qquad+y(2(a-\varepsilon)-6a\varepsilon+3\varepsilon^2+4a\varepsilon^2)+a\varepsilon(1-\varepsilon).
\end{split}    
\end{equation}
Then $\eta_1^{\boldsymbol{m}}=m\tilde{\eta}_1$,
and
\begin{align}\label{eq: varphi12}
\varphi_{1,1}&=\frac{(1-\varepsilon)a}{(1-\varepsilon)\tilde{\eta}_1+\varepsilon(1-\tilde{\eta}_1)}, \qquad \varphi_{1,2}=\frac{\varepsilon(1-a)}{\varepsilon \tilde{\eta}_1+(1-\varepsilon)(1-\tilde{\eta}_1)},\\
\varphi_{2,1}&=\frac{\varepsilon a}{(1-\varepsilon)\tilde{\eta}_1+\varepsilon(1-\tilde{\eta}_1)}, \qquad\varphi_{2,2}=\frac{(1-\varepsilon)(1-a)}{\varepsilon \tilde{\eta}_1+(1-\varepsilon)(1-\tilde{\eta}_1)},\\
\varphi_1&=\varphi_{1,1}+\varphi_{1,2},\qquad \qquad \quad \, \qquad\varphi_2=\varphi_{2,1}+\varphi_{2,2}.
\end{align}
 \end{ex}
 \begin{rk}[Small-$\varepsilon$-expansion of $\eta_1$]
 In the set-up of Example \ref{cor: parameters two-colored}, the perturbation Ansatz $y(\varepsilon)=A+B\varepsilon+O(\varepsilon^2)$ for some $A,B$
 around the boundary case $\varepsilon=0$ gives,
 as $\varepsilon \downarrow 0$,\[\tilde{\eta}_1(\varepsilon)=a+\left(a-\frac{1}{2}\right)\varepsilon+O(\varepsilon^2).\]
 \end{rk}
\begin{figure}
    \centering
    \includegraphics[width=0.7\linewidth]{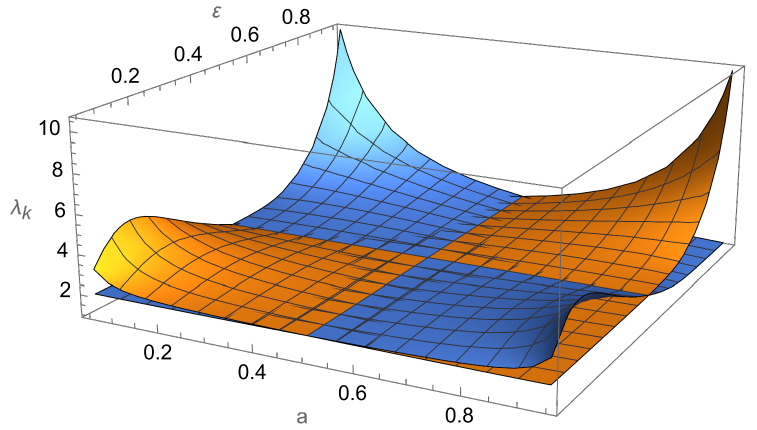}
    \caption{The values of the exponents $\lambda_1-1=\frac{2}{\varphi_1}$ (orange) and $\lambda_2-1=\frac{2}{\varphi_2}$ (blue) in the class-dependent power-law bounds in Corollary \ref{cor: exponents} for the case $m_1=m_2 \in \N$ in the set-up of Example \ref{cor: parameters two-colored}}.
    \label{fig:lambda12}
\end{figure}
\begin{figure}[t]
    \centering
    \subfloat[Homophilic setting: $\varepsilon=0.3$]{
\includegraphics[width=0.48\linewidth]{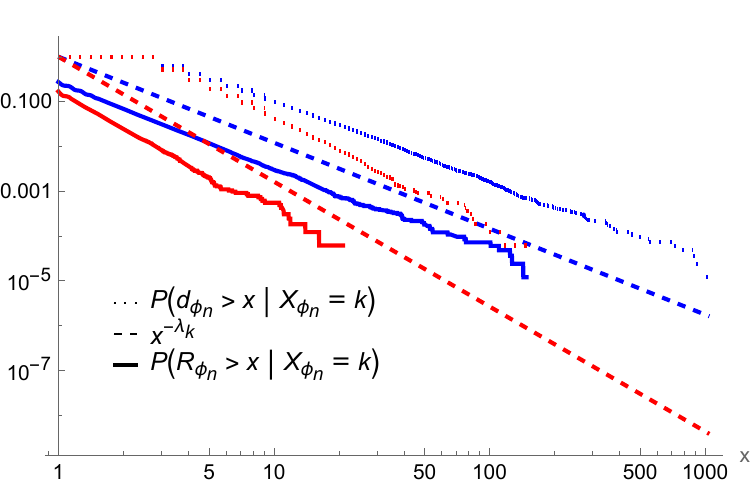}    
    }
\subfloat[Heterophilic setting: $\varepsilon=0.7$]{    \includegraphics[width=0.48\linewidth]{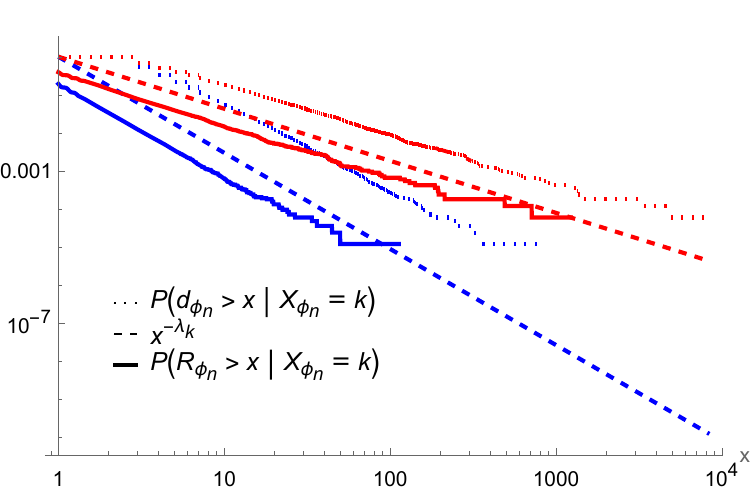}}
    \caption{Two simulations of a two-type PAM with $n+n_0=10^5$ nodes, $m_1=m_2=3$ and $\mu_1=\frac{1}{6}$, at different values of $\varepsilon$ (propensity to connect to a vertex of different color, see \eqref{eq: symmetricKappa}). Red color denotes quantities corresponding to vertices of class one (the minority) and blue color denotes class two (the majority).
 The thick lines correspond to the right-tails of the graph-normalized PageRanks (with $c=0.85$), the thinner dotted lines to the right-tails of the degrees and the dashed lines to the power laws with exponents $\lambda_{1,2}=\frac{2}{\varphi_{1,2}}-1$ with $\varphi_{1,2}$ as in Corollary \ref{cor: exponents} and \eqref{eq: varphi12}.}
    \label{fig: SimulationOFTwoTypePAM}
\end{figure}
\subsection{Fairness: Sampling uniformly, by PageRank, and by degree}\label{sec: small-c approximation}
Consider the graph-normalized total PageRank of the minority and majority vertices, respectively, given by
\begin{equation}
\begin{split} 
R_\mathrm{mino}(G_n)&=\sum_{v \in V(G_n) \colon X_v=1}R_v(G_n) \qquad \text{ and} \qquad
R_\mathrm{majo}(G_n)=\sum_{v \in V(G_n) \colon X_v=2}R_v(G_n). \end{split} \end{equation}
\paragraph{Small-$c$ approximation of the total minority PageRank.}
Iterating the PageRank equation once gives
\begin{equation}\label{eq: PRinClass}
 \begin{split}
R_\mathrm{mino}(G_n)&=\sum_{v \in V(G_n) \colon X_v=1}\left((1-c)+c\sum_{w \sim v}\frac{R_w(G_n)}{d_w(G_n)}\right)\cr
&=\sum_{v \in V(G_n) \colon X_v=1}\left((1-c)+c\sum_{w \sim v}\frac{1}{d_w(G_n)}\left((1-c)+c\sum_{\ell \sim w}\frac{R_{\ell}(G_n)}{d_\ell(G_n)}\right)\right).
 \end{split}   
\end{equation}
With the notations 
    \[N_\mathrm{mino}(G_n):=\#\{v \in V(G_n) \mid X_v=1 \} \text{ and } d^\mathrm{mino}_v(G_n):=\#\{w \in V(G_n) \mid w \sim v \text{ and } X_w=1\}
    \]
rearranging terms leads to 
\begin{equation}\label{eq: Rmino Expansion}
\begin{split}
R_\mathrm{mino}(G_n)&=(1-c)N_\mathrm{mino}(G_n)+c(1-c)\sum_{v \in V(G_n)}\frac{d_v^\mathrm{mino}(G_n)}{d_v(G_n)}+c^2\mathbbm{L}_c(G_n)\cr
&=(1-c)N_\mathrm{mino}(G_n)+c(1-c)n\mathbb{E}\Big[\frac{d_{\phi_n}^\mathrm{mino}}{d_{\phi_n}} \mid G_n\Big]+c^2\mathbbm{L}_c(G_n),
\end{split}    
\end{equation}
where $\phi_n$ is uniformly chosen from $V(G_n)$, and the remainder term $\mathbbm{L}_c$ is given by
\begin{equation*}
\begin{split}
\mathbb{L}_c&=\sum_{\ell \in V(G_n)}R_\ell(G_n)\sum_{v \in V(G_n) \colon X_v=1}\left(P(G_n)\right)^2_{\ell v} 
\end{split}
\end{equation*}
with $P_{\ell v}(G_n)=\frac{1}{d_\ell}1_{\ell \sim v}$ denoting the random walk transition matrix of $G_n$. 
We will now compare the large-$n$ asymptotics of the expected total PageRank of minority vertices $R_\mathrm{mino}(G_n)/(n+n_0)$
to $\mu_1$, which is the limiting proportion of minority vertices in the graph, for small values of the damping factor $c$. Proposition \ref{pro: MinorityFraction} below provides a first-order approximation in $c$ of this limit. 
Define
\begin{equation}\label{eq: E 1 door Dphi pro}
A_{k,\boldsymbol{m},\boldsymbol{\varphi}}:=\frac{2\Gamma(m_k+\frac{2}{\varphi_k^{\boldsymbol{m}}})}{\varphi_k^{\boldsymbol{m}} \Gamma(m_k)\Gamma(1+\frac{2}{\varphi_k^{\boldsymbol{m}}})}\left(\psi_1(1+\frac{2}{\varphi_k^{\boldsymbol{m}}})-\sum_{r=1}^{m_k-1}\frac{\beta(r,1+\frac{2}{\varphi_k^{\boldsymbol{m}}})}{r}\right),
\end{equation}
where $\beta(x,y):=\frac{\Gamma(x)\Gamma(y)}{\Gamma(x+y)}$ is the beta function and $\psi_1(z):=\frac{d^2}{dz^2}\log(\Gamma(z))$ denotes the \textit{trigamma} function.
\begin{pro}[Asymptotics of small-$c$ approximation]\label{pro: MinorityFraction}
For the two-element space $S=\{1,2\}$, any positive $\kappa: S \times S \to (0,\infty) $ and any $\mu_1 \in (0,1)$, with the quantities $(\varphi_{k,l}^{\boldsymbol{m}})_{k,l \in \{1,2\}}$ and $\eta_1^{\boldsymbol{m}}$ defined in \eqref{eq: LyapunovFunction} and \eqref{eq: varphik}, almost surely as $n\rightarrow \infty,$ 
\begin{equation}\label{eq: minority fraction pro}
\begin{split}
\mathbb{E}\Big[\frac{d_{\phi_n}^{\mathrm{mino}}}{d_{\phi_n}} \mid G_n\Big] \stackrel{n \rightarrow \infty}{\rightarrow} B^{\mathrm{mino}}_{k,\boldsymbol{m},\boldsymbol{\varphi}}:=\sum_{k=1}^2A_{k,\boldsymbol{m},\boldsymbol{\varphi}}\left(\eta_1^{\boldsymbol{m}}\varphi_{1,k}^{\boldsymbol{m}}-m_k\mu_k\frac{\varphi^{\boldsymbol{m}}_{k,1}}{\varphi^{m}_{k}}\right)+\mu_k\frac{\varphi^{\boldsymbol{m}}_{k,1}}{\varphi^{\boldsymbol{m}}_{k}}.
\end{split}
\end{equation}
In particular, by \eqref{eq: Rmino Expansion}, as $c \downarrow 0$,
\begin{equation}\label{eq: small-c final asymptotics}
\begin{split}
&\frac{R_\mathrm{mino}(G_n)}{n+n_0} \stackrel{n \to \infty}{\rightarrow}\mu_1+c\left(B^{\mathrm{mino}}_{k,\boldsymbol{m},\boldsymbol{\varphi}}-\mu_1\right)+O(c^2) \quad \text{almost surely}.
\end{split}    
\end{equation}
\end{pro}
Figure \ref{fig: PageRankShares} compares theoretical prediction of \eqref{eq: small-c final asymptotics} with a simulation result. 
\begin{con}\label{conj: minority pr}
Consider the two-colored PAM with $m_1=m_2=m$,
\begin{equation*}
\kappa(1,2)=\kappa(2,1)=\varepsilon,\qquad
 \kappa(1,1)=\kappa(2,2)=1-\varepsilon,
 \end{equation*}
 and probability $\mu_1$ of being a minority vertex. Then, for every $c>0$ sufficiently small, there is some $\delta(c)>0$ such that, for every $m \in \N$,
\begin{itemize}
     \item[$\rhd$] $\limsup_{n \to \infty}\frac{R_\mathrm{mino}(G_n)}{n+n_0}<\mu_1$ \ almost surely for every $0<\varepsilon<\frac{1}{2}-\delta(c)$ and $\mu_1 \in (0,\frac{1}{2});$ 
 \item[$\rhd$] $\liminf_{n \to \infty}\frac{R_\mathrm{mino}(G_n)}{n+n_0}>\mu_1$ \ almost surely for every $\varepsilon>\frac{1}{2}+\delta(c)$ and $\mu_1 \in (0,\frac{1}{2})$. 
  \end{itemize}
 In words, for sufficiently large graphs, the proportion of the total minority PageRank among the total PageRank $n+n_0$ is lower than the proportion of minority vertices in the network when the attachment kernel $\kappa$ is homophilic, while it is larger when $\kappa$ is  heterophilic.
\end{con}
\begin{figure}[t]
\centering
\subfloat[Homophilic setting: $\varepsilon=0.3$]{\includegraphics[width=0.48\textwidth]{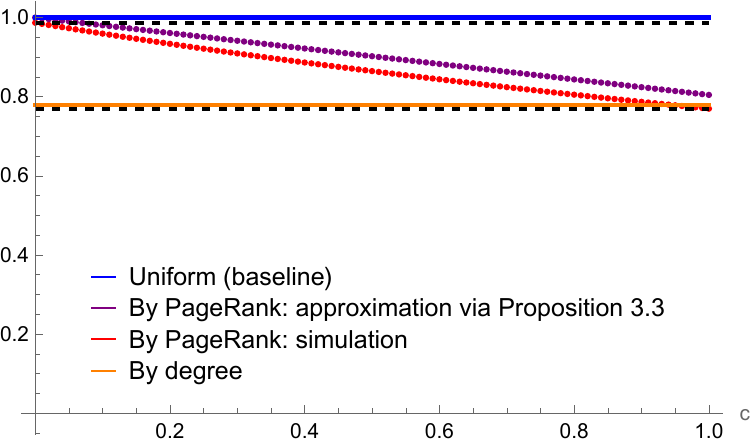}}\hfill
\subfloat[Heterophilic setting $\varepsilon=0.7$]{\includegraphics[width=0.48\textwidth]{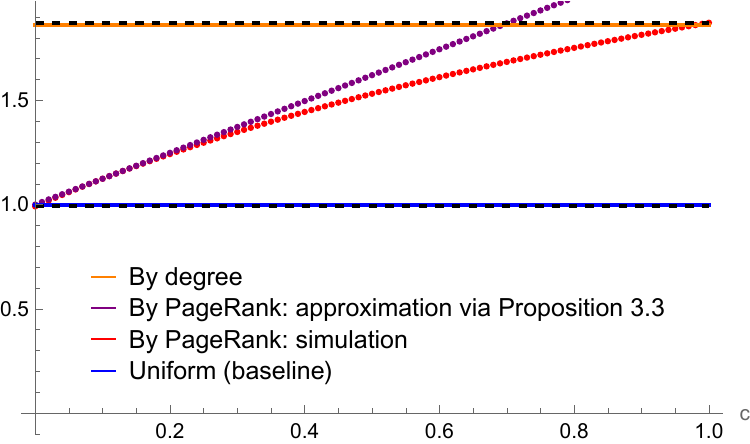}}\hfill
\caption{A comparison of the large-graph limiting probabilities to sample a minority vertex by uniform sampling (horizontal blue baseline at height 1) and by degree (horizontal orange line at level $\eta_1^{\boldsymbol{m}}/\mu_1$) at $\mu_1=\frac{1}{6}$.
In addition, the purple-colored curve shows the $c$-dependent $\frac{1}{\mu_1}$-rescaled theoretical limit for sampling by PageRank of the fraction obtained from ignoring the $O(c^2)$-contribution in \eqref{eq: small-c final asymptotics}, and inserting the expressions from Example \ref{cor: parameters two-colored}. For comparison, the red curve shows the value of $(R_\mathrm{mino}(G_n)/(n+n_0))/\mu_1$ obtained from sampling a network of size $n+n_0=10^5$ with $m_1=m_2=3$, and then computing the $c$-dependent values of $R_\mathrm{mino}(G_n)$ for $101$ equidistant points in $[0,1]$. 
The dashed horizontal lines show the proportion of minority vertices in the sampled network rescaled by a factor of $1/\mu_1$ and the proportion of the degree of minority vertices among the total degree rescaled by a factor of $1/\mu_1$.
}
\label{fig: PageRankVsC}
\end{figure}

\begin{figure}[t]
\centering
\subfloat[$m=1$]{\includegraphics[width=0.3\textwidth]{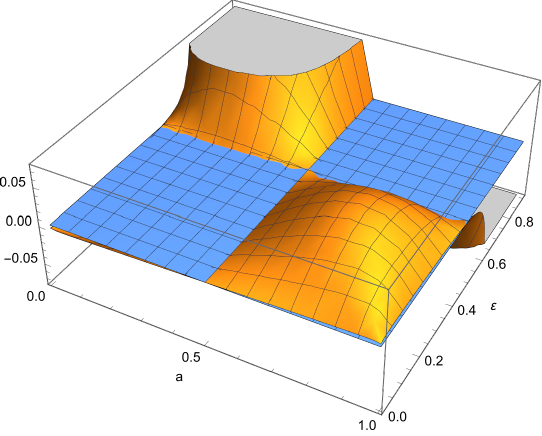}}\hfill
\subfloat[$m=3$]{\includegraphics[width=0.3\textwidth]{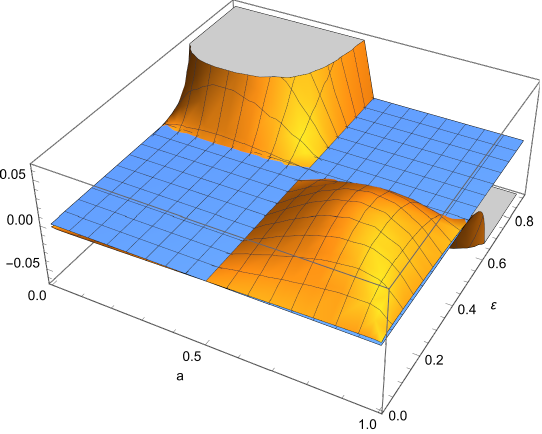}}\hfill
\subfloat[$m=20$] {\includegraphics[width=0.3\textwidth]{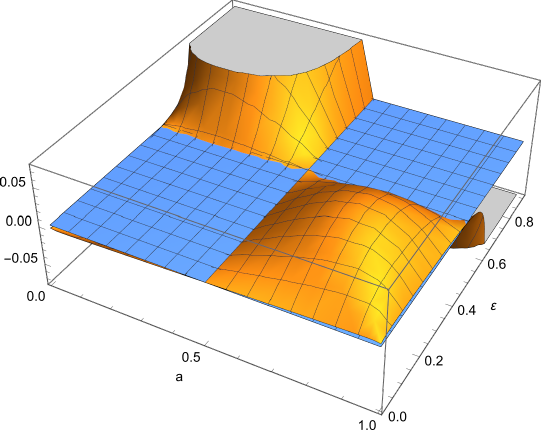}}
\caption{The first-order term $B^\mathrm{mino}_{k,\boldsymbol{m},\boldsymbol{\varphi}}-\mu_1$ in the small-$c$ expansion \eqref{eq: small-c final asymptotics} in Proposition \ref{pro: MinorityFraction} plotted as a function of $a=\mu_1$ and $\varepsilon=\kappa(1,2)=\kappa(2,1)$ at different values of $m=m_1=m_2$. The blue plane is at height $z=0$. Taking $c>0$ sufficiently small to control the $O(c^2)$ remainder in \eqref{eq: small-c final asymptotics}, the plots give rise to Conjecture \ref{conj: minority pr}.
It is known that $\boldsymbol{\eta}^{\boldsymbol{m}}$ and the $\varphi_{k}^{\boldsymbol{m}}$s are independent of $m$ in the case $m_1=m_2=m$ (cf. Corollary \ref{cor: parameters two-colored} and Remark 6 in \cite{AnBaBhPi26}). The plot suggests that in this case a similar independence of $m$ also holds for $B^\mathrm{mino}_{k,\boldsymbol{m},\boldsymbol{\varphi}}-\mu_1$.}
\label{fig: PageRankShares}
\end{figure}
    

\section{Proofs}\label{sec: proofs}
\subsection{Proof of Theorem \ref{thm: Power-law bounds}}
\subsubsection{Proof of power-law upper bound in Theorem \ref{thm: Power-law bounds}}\label{subsec: upper bound}
From \cite{AnBaBhPi26} we have the following result on the limiting degree distribution, which is a generalization from \cite{Jo13} to the case of color-dependent out-degrees:
For $k \in S$, let $\tilde{Y}^{\boldsymbol{m},(n)}_k(G_n)$ denote the total degree of color-$k$ vertices in $G_n$ rescaled by twice the number of vertices in $G_n$, i.e.,
\begin{equation}\label{eq: Y definition}
\tilde{Y}^{\boldsymbol{m},(n)}_k:=\tilde{Y}^{\boldsymbol{m},(n)}_k(G_n):=\frac{1}{2(n+n_0)}\sum_{v \in V(G_n)}\mathrm{1}\{X_v=k\}d_{v}(G_n), \quad k \in S,
\end{equation}
where in what follows we will omit the argument $G_n$ in the notation.
Introducing $\tilde{Y}^{\boldsymbol{m},(n)}_k$ allows for rewriting the attachment rule \eqref{eq: AttachmentRule} as follows:
\begin{equation}\label{eq: AttachmentEmpirical}
\mathbb{P}(v^{\sss(j)}_{n+1}\conn v \mid G_n,\ X_{n+1}=x_{n+1})=\frac{\kappa(X_v,x_{n+1})d_{v}(G_n)}{2(n+n_0)\sum_{k \in S}\kappa(k,x_{n+1})\tilde{Y}^{\boldsymbol{m},(n)}_k}. 
\end{equation}
We rely on the following key results proved in \cite{AnBaBhPi26} and \cite{HevHLi25a}:
\begin{pro}[{\cite[Lemma 4.1 and Theorem 4.4]{AnBaBhPi26}} and {\cite[Theorem 1.1]{HevHLi25a}}]\label{pro: Jordan}
~
\begin{enumerate}[(a)]
    \item The sequence of random vectors $(\boldsymbol{\tilde{Y}}^{\boldsymbol{m},(n)})_{n \in \N}$ converges almost surely to $\boldsymbol{\eta}^{\boldsymbol{m}}$ in \eqref{eq: LyapunovFunction}.
    \item For $k \in S$, recall $\varphi_k^{\boldsymbol{m}}$ in \eqref{eq: varphik}. 
    Then, for any $d \in \N_0$ and color $k \in S$, and for $\phi_n$ chosen uniformly at random from $V(G_n)$,  \begin{equation}\label{eq: LimitingDegDistribution}\mathbb{P}(d_{\phi_n}(G_n)=d \mid X_{\phi_n}=k) \stackrel{n \rightarrow \infty}{\longrightarrow} \frac{2\Gamma(m_k+\frac{2}{\varphi_k^{\boldsymbol{m}}})}{\varphi_k^{\boldsymbol{m}} \Gamma(m_k)}\frac{\Gamma(d)}{\Gamma(d+\frac{2}{\varphi_k^{\boldsymbol{m}}}+1)}.\end{equation}
    \item For any $r \geq 0$ and color $k \in S$,
    \begin{equation*}
     \limsup_{n \rightarrow \infty}\mathbb{P}(R_{\phi_n}(G_n)>r \mid X_{\phi_n}=k) \leq  \sum_{s=\lceil r \rceil}^\infty\frac{2\Gamma(m_k+\frac{2}{\varphi_k^{\boldsymbol{m}}})}{\varphi_k^{\boldsymbol{m}} \Gamma(m_k)}\frac{\Gamma(s)}{\Gamma(s+\frac{2}{\varphi_k^{\boldsymbol{m}}}+1)}.
    \end{equation*}
\end{enumerate}
\end{pro}
\begin{proof}
The first two statements are direct copies of the respective statements in \cite{AnBaBhPi26}.
For the third statement, note that by \cite[Theorem 1.1]{HevHLi25a}, PageRank in {\em any undirected graph} $G$ is pointwise bounded from above by the degree, i.e.,
\begin{equation}
R_v(G)\leq d_v(G) \quad \text{ for all vertices }v\in V(G). 
\end{equation}
In particular, the asymptotics in \eqref{eq: LimitingDegDistribution} directly implies a power-law upper bound for the limiting PageRank distribution at a uniformly chosen vertex given its color.
Now, Proposition \ref{pro: Jordan}(c) is nothing but the power-law upper bound in Theorem \ref{thm: Power-law bounds}.
\end{proof}
\begin{rk}[Interpretation of $(\eta^{\boldsymbol{m}}_k)_{k \in S}$ and $(\varphi_k^{\boldsymbol{m}})_{k \in S}$]\label{rk: interpretation of nu}
The vectors $(\eta^{\boldsymbol{m}}_k)_{k \in S}$ and $(\varphi_k^{\boldsymbol{m}})_{k \in S}$ have the following interpretations:
\begin{enumerate}[(a)]
    \item By Proposition \ref{pro: Jordan} (a), $\eta^{\boldsymbol{m}}_k$ describes the limiting proportion of edge ends of color $k$ (cf.\ \cite{Jo13}).   
\item By the attachment rule \eqref{eq: AttachmentEmpirical} and Proposition \ref{pro: Jordan}, for any $n \in \N$, $k,l \in S$ and $h \in [m_l],$ 
\begin{equation*}
\begin{split}
& \mathbb{P}(v_{n+1}^{\sss(h)} \text{ connects to vertex of color }k \text{ in } G_n \mid X_{n+1}=l)\cr 
&=\mathbb{E}\Big[\mathbb{P}(v_{n+1}^{\sss(h)} \text{ connects to vertex of color }k \text{ in } G_n \mid G_n,X_{n+1}=l)\mid X_{n+1}=l\Big]\cr 
&=\mathbb{E}\Big[\frac{\kappa(k,l) \tilde{Y}^{\boldsymbol{m},(n)}_k}{\sum_{l' \in S}\kappa(l',l)\tilde{Y}^{\boldsymbol{m},(n)}_{l'}}\mid X_{n+1}=l\Big]\stackrel{n \rightarrow \infty}{\longrightarrow} \frac{1}{\mu_lm_l}\eta^{\boldsymbol{m}}_k \varphi_{k,l}^{\boldsymbol{m}}.
\end{split}
\end{equation*}
\end{enumerate}
\end{rk}

\subsubsection{Proof of power-law lower bound in Theorem \ref{thm: Power-law bounds}}\label{subsec: lower bound}
We will make use of the known local limit for the in-component in the directed version of the multi-type PAM to prove the power-law lower bound for the present undirected version of the model.
In particular, the key ideas behind \cite[Proposition 1.8 and Corollary 1.9]{HevHLi25a} for sequences of undirected graphs are flexible enough to be adapted to the present case using the result of \cite[Theorem 4.3]{AnBaBhPi26} regarding the a.s.-local limit for the \textbf{directed} marked PAM.
\paragraph{A brief recap on local convergence.}
Let us begin with a brief recap on local convergence. For a more detailed overview we refer to \cite[Chapter 2 and Section 9.2]{vH24}.\\ First assume that $(\mathcal{G}_n)_{n \in \N}$ is any sequence of finite \textbf{undirected} graphs. Given any radius $s \in \N$ and any deterministic rooted graph $(\mathcal{H},\rho)$,  consider the random set 
\begin{equation}\label{eq: numberOfNeighborhoods}
\mathcal{N}_{(\mathcal{H},\rho),s}^{(n)}:=\{v \in V(\mathcal{G}_n) \mid B_{s}(\mathcal{G}_n,v) \cong (\mathcal{H},\rho) \} 
\end{equation}
of vertices $v$ for which the closed ball of radius $s$ around $v$ (with respect to the usual graph distance), viewed as a rooted graph with root $v$, is equivalent to $(\mathcal{H},\rho)$ under graph-isomorphisms that preserve the root.
Then we say that $(\mathcal{G}_n)_{n \in \N}$ converges \textbf{locally} (weakly/in probability/almost surely) to a random locally finite connected rooted graph $(\mathcal{G}^\star,\rho^\star)$ when, for all $s \in \N$ and $(\mathcal{H},\rho)$,
\begin{equation}\label{eq: localConvergence}
\frac{\#\mathcal{N}_{(\mathcal{H},\rho),s}^{(n)}}{\#V(G_n)} \stackrel{n \to \infty}{\rightarrow}\mathbb{P}(B_s(\mathcal{G}^\star,\rho^\star) \cong (\mathcal{H},\rho)) \quad \text{(weakly/in probability/almost surely).}\end{equation}
Next, for any sequence of finite \textbf{directed} graphs $(\vec{\mathcal{G}}_n)_{n \in \N}$, where each vertex carries an attribute or mark, the corresponding notion of local convergence to a locally finite rooted directed graph $(\vec{\mathcal{G}}_n,\rho)$ with respect to \textit{vertex-marked incoming neighborhoods} is obtained as follows:
In \eqref{eq: numberOfNeighborhoods} and \eqref{eq: localConvergence}, replace the ball $B_s(\mathcal{G}_n,v)$ by the \textit{incoming neighborhood} $U_s(\vec{\mathcal{G}}_n,v)$ consisting of all vertices from which the vertex $v$ can be reached by following a directed path of length at most $s$, where $\cong$ denotes equivalence under isomorphisms that preserve both the root and the marks of all vertices.\\
In the framework of this paper, the above mentioned marks are the colors of the vertices, which determine the out-degrees of the respective vertices.

\paragraph{Local limit for colored directed PAM as in \cite{AnBaBhPi26}.}
Consider the colored \textbf{directed} multi-type PAM $(\vec{G}_n)_{n \in \N}$ with set of colors $S=\{1,2,\ldots,q\}$, which we interpret as vertex marks, and fixed out-degree vector $\boldsymbol{m}$. For the directed version, the attachment rule is as in \eqref{eq: AttachmentRule}, but edges are now regarded as directed from new to old vertices.
Aligning as much as possible with the notations used in the previous sections, let us briefly summarize the key statement of \cite[Theorem 4.3]{AnBaBhPi26}.
\begin{defn}[Definition 4.2 in \cite{AnBaBhPi26}]
The characterization of the local limit of the colored directed PAM is as follows: 
\begin{itemize}
    \item[$\rhd$] For any color $k \in S$, let $\left(\mathrm{BP}_k(t)\right)_{t \geq 0}$, denote the continuous-time multi-type branching process started with an ancestor (root) $\phi$ of color $k$, where each individual of color $k' \in S$ reproduces independently at times following the law of a Markovian pure birth process $\xi_{k'}(\cdot)$ with birth rate a color-$l$ offspring at time $t$ given by $(\xi_{k'}(t)+m_{k'})\varphi_{k',l}$.
    \item[$\rhd$] For any (possibly random) color $X$ and time $T$, denote the progeny tree of $\mathrm{BP}_X(T)$ rooted in the ancestor, viewed as a directed, marked, rooted graph with each vertex marked by its individual color by $\mathcal{T}_X(T)$.
\end{itemize}
\end{defn}
\begin{thm}[{\protect \cite[Theorem 4.3]{AnBaBhPi26}}]\label{thm: locallimitdirectedmtypepam}
For fixed $\kappa$ and $\mu$,
\[\vec{G}_n \stackrel{a.s.-locally}{\longrightarrow}\mathcal{T}_X(\tau),\]
where $X \sim \mu$ and $\tau$ is an independent exponential random variable with rate $\lambda=2$.
Here, $\stackrel{a.s.-locally}{\longrightarrow}$ means local convergence almost surely with respect to vertex-marked incoming neighborhoods.
\end{thm}

\paragraph{Truncated first-order undirected PageRank on directed graph.}
We want to obtain a lower bound on the PageRank in the undirected multi-type PAM by restricting to first-order terms associated with incoming edges. The latter quantity is a function of the incoming neighborhood of radius two of the considered vertex and its color and thus its limit in distribution is determined by the known directed local limit. 
Throughout the remainder of this paper we denote the in-degree of any vertex $v$ by $d_v^-$ , and the out-degree of vertex $v$ by $d_v^+$.
We next define this truncated first-order lower bound:
 \begin{defn}[Truncated first-order undirected PageRank on directed graph]\label{def: TruncatedPageRank}
 {\normalfont \rmfamily} 
We define the \textbf{truncated undirected PageRank on the directed multi-type PAM} $\vec{G}_n$ by
\begin{equation*}
R^{1,\mathrm{dirTr}}_v(\vec{G}_n):=c(1-c)\sum_{\substack{u \in V(\vec{G}_n)\\u \rightarrow v}}\frac{1}{m_{X_u}+d_u^-},    \quad v \in \vec{G}_n,
\end{equation*}
where $u \rightarrow v$ indicates that the sum runs over all vertices $u$ that have an edge directed to $v$.
\end{defn}
Clearly, for every $n \in \N$ and vertex $v$,
\begin{equation}\label{eq: PageRank-DirVsUndir}R_v(G_n) \geq R^{1,\mathrm{dirTr}}_v(\vec{G}_n),
\end{equation}
where we regard $G_n$ as the result of the natural embedding of the set of directed rooted graphs into the set of rooted undirected graphs, by copying the set of vertices and replacing all oriented edge by  unoriented edges. The directed local convergence in Theorem \ref{thm: locallimitdirectedmtypepam} implies the following convergence:
\begin{pro}\label{pro: LowerBound}
The following holds true:
\begin{enumerate}[(a)]
    \item The sequence $(R^{1,\textup{dirTr}}_{\phi_n}(\vec{G}_n))_{n \in \N}$ converges in distribution to
\[R^{1,\textup{dirTr}}_\phi(\mathcal{T}_X(\tau)):=c(1-c)\sum_{u \rightarrow \phi}\frac{1}{m_{X_u}+d_u^-}.\]
\item Fix $k \in S$. Conditionally on $X_{\phi_n}=k$, the sequence $(R^{1,\textup{dirTr}}_{\phi_n}(\vec{G}_n))_{n \in \N}$ converges in distribution to
\[R^{1,\textup{dirTr}}_\phi(\mathcal{T}_k(\tau)):=c(1-c)\sum_{u \rightarrow \phi}\frac{1}{m_{X_u}+d_u^-},\]
where the r.h.s. refers to the process $\mathcal{T}_X(\tau)$ conditioned on $X=k$.
\item For $k \in S$, there exists some finite $\beta$ such that, as $r \rightarrow \infty$,
\[\mathbb{P}(R^{1,\textup{dirTr}}_\phi(\mathcal{T}_X(\tau))>r \mid X_\phi=k) \geq  (1+o(1)) \;\mathbb{P}(d_{\phi}^->\beta r \mid X_\phi=k).\]
\end{enumerate}
In particular, conditionally on $X_\phi=k$, the random variable $R^{1,\textup{dirTr}}_\phi(\mathcal{T}_X(\tau))$ satisfies a power law with the same $k$-dependent exponent as the (in-)degree.
\end{pro}
\begin{proof}[Proof of Proposition \ref{pro: LowerBound}]
For any fixed $r \in \R$, the indicator $\mathbf{1}\{R^{1,\mathrm{dirTr}}_\phi > r\}$ is a bounded continuous function on the space of marked rooted directed connected graphs up to root-preserving isomorphism (cf.\ \cite[Section 2.2]{vH24}) which only depends on the incoming neighborhood of finite radius two. Hence, the stated convergence follows from Theorem \ref{thm: locallimitdirectedmtypepam} and \cite[Theorem 2.15]{vH24}, which relates local convergence in the sense of \eqref{eq: localConvergence} to convergence of (conditional) expectations of bounded continuous functions. This proves part (a).
To prove part (b), we consider the indicator $\mathbf{1}\{X_{\phi}=k \text{ and }R^{1,\mathrm{dirTr}}_\phi > r\},$ which again is a bounded continuous function on the space of marked rooted directed connected graphs up to to root-preserving isomorphism.\\
The proof of part (c) is more intricate in that we need to control degree correlations, and relies on Lemmas \ref{lem: PageRankLowerBound}--\ref{lem: lowerbound} below.
In particular, Lemma \ref{lem: PageRankLowerBound} provides a sufficient criterion  that ensures that the distribution of $R^{1,\textup{dirTr}}_\phi(\mathcal{T}_k(\tau))$ has a power-law lower bound with the same exponent as that of the distribution of $d_\phi^-$. This criterion is based on controlling the degrees of the children of $d_\phi$ when conditioning on $d_\phi^-=r$ and letting $r \to \infty$. Note that as $r$ increases, also the expected value of the time $\tau$ at which the process is stopped increases. To address this issue, for given $t>0$ the Lemmas \ref{lem: lateStop} and \ref{lem: exclusiveFraction} control the proportion of those children of $\phi$ that are born in the time-interval $[\tau-t,\tau]$. The distribution of the degrees of these children then can be uniformly bounded stochastically. Finally, Lemma \ref{lem: lowerbound} verifies the criterion of Lemma \ref{lem: PageRankLowerBound}. We will postpone the proofs of Lemmas \ref{lem: PageRankLowerBound}, \ref{lem: lateStop},\ref{lem: exclusiveFraction} and \ref{lem: lowerbound} to the end of this section and first finish the proof of Proposition \ref{pro: LowerBound}.
\\ 
To formalize the reasoning outlined above, for $\alpha>0$, we first define
   \begin{equation*}
        \dphigeqalphadir=\#\{v \in \mathcal{T}_X(\tau)\mid v \rightarrow \phi \text{ and } d^-_v \geq \alpha\}
    \end{equation*}
and, for $k \in S$, we let $\mathbb{P}^\mathcal{T}_k(\cdot):=\mathbb{P}^\mathcal{T}(\cdot \mid X_\phi=k)$ denote the distribution of the branching process $\mathcal{T}_X(\tau)$ conditioned on the root having color $k$.
\begin{lemma}[Root-PageRank lower bound (cf.\ \cite{HevHLi25a})]\label{lem: PageRankLowerBound} 
Assume that, for $k \in S$, there exist $\alpha>0$ and $p \in(0,1)$ such that, as $r\rightarrow \infty,$
\begin{equation}\label{eq: LowerBoundMainAssumtion}
\mathbb{P}_k^{\mathcal{T}}\left( \dphigeqalphadir 
\geq pd^-_\phi \mid d_\phi^-=r\right) \stackrel{r \to \infty}{\rightarrow} 0.
\end{equation}
Then, with $\beta:=\frac{m_k+\alpha}{(1-p) c(1-c)}$, the random variable $R^{1,\textup{dirTr}}_\phi$ satisfies 
\begin{equation}\label{cond: o of dphi>K}
\mathbb{P}_k^{\mathcal{T}}(R^{1,\textup{dirTr}}_\phi>r) \geq (1+o(1))\mathbb{P}_k^{\mathcal{T}}(d_\phi^->\beta r) \quad \text{as } r \rightarrow \infty.   
\end{equation}
\end{lemma}
The proof of Lemma \ref{lem: PageRankLowerBound} will be given at the end of this section.
To verify the assumption \eqref{eq: LowerBoundMainAssumtion} in Lemma \ref{lem: PageRankLowerBound} above, we will make use of the fact that for every $k \in S$ \enquote{most} of the children of the root $\phi$ in $\mathcal{T}_k(\tau)$ are born near the end of the time interval $[0,\tau]$. In particular, consider the independent family $(\tilde{\tau}_k^{\sss(j)})_{j \in \N}$ with $\tilde{\tau}_k^{\sss(j)}\sim \mathrm{Exp}((j+m_k)\varphi_k^{\boldsymbol{m}})$. Here, $\tilde{\tau}_k^{\sss(j)}$ describes the waiting time in between the births of the $j-1$th child and the $j$th child of $\phi$ conditionally on $X_\phi=k$.
We first note the following:
\begin{lemma}\label{lem: lateStop}
 For every $t>0, \alpha>0$ and $k \in S$,
\begin{equation*}
\begin{split}\limsup_{r \to \infty}\mathbb{P}_k^{\mathcal{T}}\left( \dphigeqalphadir 
\geq pd^-_\phi \mid d_\phi^- =r\right) \leq \limsup_{r \to \infty}\mathbb{P}_k^{\mathcal{T}}\left( \dphigeqalphadir 
\geq pd^-_\phi \mid d_\phi^- =r, \ \tau = \sum_{j=1}^r \tilde{\tau}_k^{\sss(j)}+t\right)
\end{split}
\end{equation*}
\end{lemma}
The proof of Lemma \ref{lem: lateStop} will be given at the end of this section.
Next, for fixed $\tilde{t}>0$, define
\begin{equation}\label{eq: ellDefinition}
 \ell_{k,\tilde{t}}^{\sss(r)}:=\inf\{l \in \{1,2,\ldots,r\} \mid \sum_{j=l}^{r} \tilde{\tau}_k^{\sss(j)} \leq \tilde{t} \} \in \N \cup \{\infty\}.   
\end{equation}
The idea behind \eqref{eq: ellDefinition} is that, conditionally on $d_\phi^-=r$, we will take the first $\ell_{k,\tilde{t}}^{\sss(r)}-1$ children of $\phi$ out of consideration, because the defining property of $\ell_{k,\tilde{t}}^{\sss(r)}$ and Lemma \ref{lem: lateStop} allow us to uniformly bound the time to produce offspring for each of the later born children of $\phi$. To follow this approach, we first need to control the asymptotics of the ratio $\frac{\ell_{k,\tilde{t}}^{\sss(r)}}{r}$.
\begin{lemma}\label{lem: exclusiveFraction}
For given $k \in S$ fix $\tilde{t}>2/\varphi_k^{\boldsymbol{m}}$. Then there is some $\tilde{p}_{k,\tilde{t}}\in (0,1)$ such that  $\mathbb{P}\left(\ell_{k,\tilde{t}}^{\sss(r)}/r >\tilde{p}_{k,\tilde{t}} \right) \stackrel{r \to \infty}{\rightarrow} 0$.  
\end{lemma}
The proof of Lemma \ref{lem: exclusiveFraction} will be given at the of of this section.
Finally, we verify the Assumption \ref{eq: LowerBoundMainAssumtion}:
\begin{lemma}[Assumption implying \eqref{eq: LowerBoundMainAssumtion}]
\label{lem: lowerbound}
Fix $k \in S$. Then, for any $\tilde{t}>2/\varphi_k^{\boldsymbol{m}}$, choose $\tilde{p}_{k,\tilde{t}} \in (0,1)$ as in Lemma \ref{lem: exclusiveFraction}. Next set $p \in (\tilde{p}_{k,\tilde{t}},1)$ and arbitrarily choose $t>0$ and $\varepsilon \in (0,p-\tilde{p})$. Finally, choose
$\alpha:=\bar{F}_{D_{k,\tilde{t}+t}}^{\leftarrow}(p-\tilde{p}-\varepsilon):=\min\{s \in \N \mid \mathbb{P}(D_{k,\tilde{t}+t} \geq s) \leq p-\tilde{p}-\varepsilon\}$ where $D_{k,\tilde{t}+t} \sim Poisson(m_k\varphi_k (\tilde{t}+t)).$  Then the assumption \eqref{eq: LowerBoundMainAssumtion} is fulfilled for any of the above choices of $p$ and $\alpha$.
In particular, for any $\beta>\frac{2(m_k+\alpha)}{(1-p)c(1-c)}$, 
\begin{equation*}
\mathbb{P}_k^{\mathcal{T}}(R_\phi^{1,\textup{dirTr}}>r) \geq (1+o(1))\mathbb{P}_k^{\mathcal{T}}(d_\phi^->\beta r) \quad \text{as }r \rightarrow \infty.  
\end{equation*}
\end{lemma}
The proof of Lemma \ref{lem: lowerbound} will be given at the of of this section.
\paragraph{Completion of the proof of Proposition \ref{pro: LowerBound}.}
Part (c) in Proposition \ref{pro: LowerBound} now follows directly from Lemma \ref{lem: lowerbound} and the assumption that $S$ is finite.
\end{proof}
\begin{proof}[Proof of the lower bound in Theorem \ref{thm: Power-law bounds}]
For any $n \in \N$, we  regard the undirected PAM $G_n$ with finite set $S$ of colors, color distribution $\mu$ and attractiveness function $\kappa$ as a result of embedding the directed PAM $\vec{G}_n$ with the same parameters into the set of undirected graphs by identifying each directed edge with an undirected one. 
First, by \eqref{eq: PageRank-DirVsUndir}, the PageRank $R_{v}(G_n)$ of any $v \in G_n$ is stochastically bounded below by $R^{1,\mathrm{dirTr}}_v(\vec{G}_n)$, the truncated directed PageRank as defined in Definition \ref{def: TruncatedPageRank}. In combination with Proposition \ref{pro: LowerBound}(c) for the tail distribution of the random variable $R^{1,\mathrm{dirTr}}_v(\vec{G}_n)$ conditioned on a class $k$, this gives, as $r \rightarrow \infty$,
\begin{equation}\label{eq: AsymptLowerBoundInDegree}
\begin{split}
\liminf_{n \rightarrow \infty} \mathbb{P}(R_{\phi_n}(G_n)>r \mid X_{\phi_n}=k) &\geq \mathbb{P}(R^{1,\mathrm{dirTr}}_\phi(\mathcal{T}_X(\tau))>r \mid X_\phi=k) \cr 
&\geq  (1+o(1)) \;\mathbb{P}_k^{\mathcal{T}}(d_{\phi}^->\beta r \mid X_\phi=k).
\end{split}
\end{equation}
Finally, to make use of the known distribution of $d_\phi^{-}+m_{X_\phi}$, which is 
\begin{equation}\label{eq: AntunesDeg}
\mathbb{P}_k^{\mathcal{T}}(d_\phi^-+m_{X_\phi}=r)=\frac{2\Gamma(m_k+\frac{2}{\varphi_k^{\boldsymbol{m}}})}{\varphi_k^{\boldsymbol{m}} \Gamma(m_k)}\frac{\Gamma(r)}{\Gamma(r+\frac{2} {\varphi_k^{\boldsymbol{m}}}+1)}    
\end{equation}
(see \cite[Theorem 4.4]{AnBaBhPi26}), we
note that for any $\gamma>\beta$ the inequality $d_v>\gamma r$ implies $d_v^->\beta r$ once $r > \frac{\gamma-\beta}{m_\text{max}}$. In this way, we obtain from \eqref{eq: AsymptLowerBoundInDegree}, as $r \rightarrow \infty$,
\begin{equation}
\liminf_{n \rightarrow \infty} \mathbb{P}(R_{\phi_n}(G_n)>r \mid X_{\phi_n}=k) \geq (1+o(1))\mathbb{P}_k^{\mathcal{T}}(d_{\phi}^-+m_{X_\phi}>\gamma r ).  
\end{equation}
Finally inserting \eqref{eq: AntunesDeg} completes the proof of the lower bound in Theorem \ref{thm: Power-law bounds}.
\end{proof}
\paragraph{Postponed proofs of Lemmas \ref{lem: PageRankLowerBound} to \ref{lem: lowerbound}}\label{par: postponed proofs}
\begin{proof}[Proof of Lemma \ref{lem: PageRankLowerBound}]
Analogously to the proof of \cite[Proposition 1.8]{HevHLi25a}, we write
\begin{align}
&\mathbb{P}_k^{\mathcal{T}}\Big(c(1-c)\sum_{j \rightarrow \phi}\frac{1}{m_k+d_v^-}>r\Big)\cr
&\geq \mathbb{P}_k^{\mathcal{T}}\Big(c(1-c)\sum_{v \rightarrow \phi}\frac{1}{m_k+d_v^-}>r,\#\{v \colon  v \rightarrow \phi \text{ and } d_v^- < \alpha\} \geq (1-p) d_\phi^-\Big)\\
&\geq \mathbb{P}_k^{\mathcal{T}}\Big(d_\phi^->\frac{m_k+\alpha}{(1-p) c(1-c)}r,\#\{v \colon v \rightarrow \phi \text{ and } d_v^- < \alpha\} \geq (1-p) d_\phi^-\Big)\nn\cr
&\geq \mathbb{P}_k^{\mathcal{T}}\Big(d_\phi^->\frac{m_k+\alpha}{(1-p) c(1-c)}r\Big)-\mathbb{P}_k^{\mathcal{T}}\Big(d_\phi^->\frac{m_k+\alpha}{(1-p) c(1-c)}r, \dphigeqalphadir  \geq p d_\phi^-\Big)\cr 
&=\mathbb{P}_k^{\mathcal{T}}(d_\phi^->\beta r )-\mathbb{P}_k^{\mathcal{T}}(d_\phi^->\beta r, \dphigeqalphadir  \geq p d_\phi^-).\nn 
\end{align}
Finally,
\begin{equation*}
\begin{split}
\frac{\mathbb{P}_k^{\mathcal{T}}\Big(d_\phi^->\beta r, \dphigeqalphadir  \geq p d_\phi^-\Big)}{\mathbb{P}_k^{\mathcal{T}}\Big(d_\phi^->\beta r \Big)}&=\sum_{j=0 }^\infty\mathbb{P}_k^{\mathcal{T}}(\dphigeqalphadir  \geq p d_\phi^- \mid d_\phi^-=j)\frac{1\{j \geq \lceil \beta r \rceil \}\mathbb{P}_k^{\mathcal{T}}(d_\phi^-=j)}{\mathbb{P}_k^{\mathcal{T}}(d_\phi^->\beta r)}    \cr 
&\leq \sup_{j \geq \lceil \beta r \rceil} \mathbb{P}_k^{\mathcal{T}}(\dphigeqalphadir  \geq p d_\phi^- \mid d_\phi^-=j) \stackrel{r \to \infty}{\rightarrow} 0,
\end{split}    
\end{equation*}
where the upper bound follows from the fact that we are summing with respect to a probability measure supported on $\N \cap [\beta r, \infty]$ and the convergence to zero then is a consequence of Assumption \eqref{eq: LowerBoundMainAssumtion}.
Thus, $\mathbb{P}_k^{\mathcal{T}}\Big(d_\phi^->\frac{m_k+\alpha}{(1-p) c(1-c)}r, \dphigeqalphadir  \geq p d_\phi^-\Big) =o(1)\mathbb{P}_k^{\mathcal{T}}\Big(d_\phi^->\frac{m_k+\alpha}{(1-p) c(1-c)}r\Big)$, which concludes the proof of Lemma \ref{lem: PageRankLowerBound}.
\end{proof}
\begin{proof}[Proof of Lemma \ref{lem: lateStop}]
We have 
\begin{equation}\label{eq: tauExcess}
\mathbb{P}_k^{\mathcal{T}}(\tau>\sum_{j=1}^r \tilde{\tau}_k^{\sss(j)}+t \mid d_\phi^-=r)=\frac{\mathbb{P}_k^{\mathcal{T}}(\sum_{j=1}^r \tilde{\tau}_k^{\sss(j)}+t \leq \tau <\sum_{j=1}^{r+1} \tilde{\tau}_k^{\sss(j)})}{\mathbb{P}_k^{\mathcal{T}}(d_\phi^-=r)} \leq \frac{\mathbb{P}_k^{\mathcal{T}}(\tilde{\tau}_{k}^{\sss(r+1)}>t)}{\mathbb{P}_k^{\mathcal{T}}(d_\phi^-=r)} \stackrel{r \to \infty}{\rightarrow} 0.  
\end{equation}
Here, the convergence follows from the fact that the numerator in \eqref{eq: tauExcess} decreases exponentially fast with increasing $r$, while the denominator decreases only polynomially fast (see \eqref{eq: LimitingDegDistribution}).
Equation \eqref{eq: tauExcess} then gives 
\begin{equation*}
\begin{split}
&\limsup_{r \to \infty}\mathbb{P}_k^{\mathcal{T}}\left( \dphigeqalphadir 
\geq pd^-_\phi \mid d_\phi^- =r\right) =\limsup_{r \to \infty}\mathbb{P}_k^{\mathcal{T}}\left( \dphigeqalphadir 
\geq pd^-_\phi, \ \tau \leq \sum_{j=1}^r \tilde{\tau}_k^{\sss(j)}+t \mid d_\phi^- =r\right)\cr
&\leq \limsup_{r \to \infty}\mathbb{P}_k^{\mathcal{T}}\left( \dphigeqalphadir 
\geq pd^-_\phi \mid d_\phi^- =r, \ \tau = \sum_{j=1}^r \tilde{\tau}_k^{\sss(j)}+t\right),
\end{split}
\end{equation*}
which finishes the proof of Lemma \ref{lem: lateStop}.
\end{proof}
\begin{proof}[Proof of Lemma \ref{lem: exclusiveFraction}]
The proof is based on calculating the first two moments of the random sum in \eqref{eq: ellDefinition} and then applying Chebyshev's inequality. We start by analyzing the deterministic versions of \eqref{eq: ellDefinition} given by
\[\bar{\ell}_{k,\tilde{t}/2}^{\sss(r)}:=\inf\{l \in \{1,2,\ldots,r\} \mid \sum_{j=l}^{r} \mathbb{E}[\tilde{\tau}_k^{\sss(j)}] \leq \tilde{t}/2 \} \in \N \cup \{\infty\}.\]
First, \begin{equation}\label{eq: ExpToHarmonic}
\mathbb{E}\Big[\sum_{j=l}^{r} \tilde{\tau}_k^{\sss(j)}\Big]=\frac{1}{\varphi_k^{\boldsymbol{m}}}\sum_{j=l}^{r}\frac{1}{j+m_k}=\frac{1}{\varphi_k^{\boldsymbol{m}}}\sum_{j=l+m_k}^{r+m_k}\frac{1}{j}.
\end{equation}
Using $\sum_{j=1}^n \frac{1}{j}=\log(n)+\gamma+\frac{1}{2n}-L_n,$
where $\gamma$ is the Euler-Mascheroni constant, and $L_n \in [0, \frac{1}{8n^2}]$ and considering the boundary cases for $L_n$ gives
\begin{equation}\label{eq: HarmonicUpperBound}
\sum_{j=l+m_k}^{r+m_k}\frac{1}{j} \leq \log(r+m_k)+
\frac{1}{2(r+m_k)}-\left(\log(l+m_k-1)+
\frac{1}{2(l+m_k-1)}-\frac{1}{8(l+m_k-1)^2} \right),
\end{equation}
so from \eqref{eq: ExpToHarmonic} we obtain that $\mathbb{E}[\sum_{j=l}^{r} \tilde{\tau}_k^{\sss(j)}] \leq \tilde{t}/2$ is guaranteed whenever the r.h.s. of \eqref{eq: HarmonicUpperBound} is smaller or equal to $\varphi_k^{\boldsymbol{m}}\tilde{t}/2$, which is equivalent to $l$ being large enough to satisfy
\begin{equation}\label{eq: tildeLUB}
\begin{split}
&\log(l+m_k-1)+\frac{1}{2(l+m_k-1)}-\frac{1}{8(l+m_k-1)^2} \geq \log(r+m_k)+\frac{1}{2(r+m_k)}-\varphi_{m_k}^{\boldsymbol{m}}\tilde{t}/2.
\end{split}
\end{equation}
In other words, the infimum of all $l \in \{1,2,\ldots,r\}$ that satisfy \eqref{eq: tildeLUB} provides an upper bound on $\bar{\ell}_{k,\tilde{t}/2}^{\sss(r)}.$
Similarly,
\begin{equation}\label{eq: HarmonicLowerBound}\sum_{j=l+m_k}^{r+m_k}\frac{1}{j} \geq \log(r+m_k)+
 \frac{1}{2(r+m_k)}-\frac{1}{8(r+m_k)^2}-\left(\log(l+m_k-1)+
\frac{1}{2(l+m_k-1)} \right),
\end{equation}
so from \eqref{eq: ExpToHarmonic} we obtain that if $\mathbb{E}[\sum_{j=l}^{r} \tilde{\tau}_k^{\sss(j)}] \leq \tilde{t}/2$ then the r.h.s.\ of \eqref{eq: HarmonicLowerBound} is necessarily smaller than or equal to $\varphi_k^{\boldsymbol{m}}\tilde{t}/2$, which is equivalent to $l$ being large enough to satisfy
\begin{equation}\label{eq: TildeLLB}
\begin{split}
&\log(l+m_k-1)+\frac{1}{2(l+m_k-1)} \geq \log(r+m_k)+\frac{1}{2(r+m_k)}-\frac{1}{8(r+m_k)^2}-\varphi_{m_k}^{\boldsymbol{m}}\tilde{t}/2.    
\end{split}    
\end{equation}
In other words, the infimum of all $l \in \{1,2,\ldots,r\}$ that satisfy \eqref{eq: TildeLLB} provides a lower bound on $\bar{\ell}_{k,\tilde{t}/2}^{\sss(r)}.$
Note that any $l$ that is large enough to fulfill \eqref{eq: tildeLUB} will also fulfill \eqref{eq: TildeLLB}.
Further, applying the inequalities $\log(x)+\frac{1}{2x}-\frac{1}{8x^2} \geq \log(x)$ and $\log(x)+\frac{1}{2x}\leq \log(x+1)$ with $x=l+m_k-1>1$ to \eqref{eq: tildeLUB} and \eqref{eq: TildeLLB} gives that
\begin{equation}\label{eq: ellUpperBound}
\bar{\ell}_{k,\tilde{t}/2}^{\sss(r)} \leq (r+m_k)e^{\frac{1}{2(r+m_k)}-\varphi_{m_k}^{\boldsymbol{m}}\tilde{t}/2}+1-m_k \quad \text{ and } 
\end{equation}
\begin{equation}\label{eq: ellLowerBound}
\bar{\ell}_{k,\tilde{t}/2}^{\sss(r)} \geq (r+m_k)e^{\frac{1}{2(r+m_k)}-\varphi_{m_k}^{\boldsymbol{m}}\tilde{t}/2}-m_k.
\end{equation}
From the assumption $\tilde{t}>\frac{2}{\varphi_{m_k}}$ and \eqref{eq: ellUpperBound} we thus obtain that $
\limsup_{r \to \infty}\bar{\ell}_{k,\tilde{t}/2}^{\sss(r)}/r<1.$
Next, we compute the variance as
\begin{equation}\label{eq: ellVariance}
\begin{split}
\mathbb{V}\Big[\sum_{j=\bar{\ell}_{k,\tilde{t}/2}^{\sss(r)}}^{r} \tilde{\tau}_k^{\sss(j)}\Big]&=\frac{1}{\varphi_k^{\boldsymbol{m}}}\sum_{j=\ell_{k,\tilde{t}}^{\sss(r)}+m_k}^{r+m_k}\frac{1}{j^2} \leq \frac{1}{\varphi_k^{\boldsymbol{m}}} \int_{\bar{\ell}_{k,\tilde{t}/2}^{\sss(r)}+m_k-1}^{r+m_k}\frac{1}{x^2} \ \text{d}x\cr 
&=\frac{1}{\varphi_k^{\boldsymbol{m}}}\left(\frac{1}{\bar{\ell}_{k,\tilde{t}/2}^{\sss(r)}+m_k-1}-\frac{1}{r+m_k}\right).  
\end{split}
\end{equation}
By \eqref{eq: ellLowerBound}, we obtain $\bar{\ell}_{k,\tilde{t}/2}^{\sss(r)}+m_k-1 \geq re^{-\varphi_{m_k}\tilde{t}/2}-1$ for all $r \in \N$. Thus, from \eqref{eq: ellVariance}, we obtain
$\mathbb{V}\Big[\sum_{j=\bar{\ell}_{k,\tilde{t}/2}^{\sss(r)}}^{r} \tilde{\tau}_k^{\sss(j)}\Big] =O(\frac{1}{r}).$
Then, Chebyshev's inequality gives
\begin{equation}\label{eq: ChebyshevForEll}
\mathbb{P}\Bigg(\sum_{j=\bar{\ell}_{k,\tilde{t}/2}^{\sss(r)}}^{r} \tilde{\tau}_k^{\sss(j)}>\tilde{t} \Bigg) \leq \mathbb{P}\Bigg(\Big\vert \sum_{j=\bar{\ell}_{k,\tilde{t}/2}^{\sss(r)}}^{r} \tilde{\tau}_k^{\sss(j)} -\mathbb{E}\Big[\sum_{j=\bar{\ell}_{k,\tilde{t}/2}^{\sss(r)}}^{r} \tilde{\tau}_k^{\sss(j)}\Big]\Big\vert > \tilde{t}/2\Bigg) \leq (\tilde{t}/2)^{-2}\mathbb{V}\Bigg[\sum_{j=\bar{\ell}_{k,\tilde{t}/2}^{\sss(r)}}^{r} \tilde{\tau}_k^{\sss(j)}\Bigg].     
\end{equation}
From \eqref{eq: ChebyshevForEll}, we conclude
\begin{equation}\label{eq: EllVsBarEll}
\mathbb{P}(\ell_{k,\tilde{t}}^{\sss(r)}>\bar{\ell}_{k,\tilde{t}/2}^{\sss(r)}) \stackrel{r \to \infty}{\rightarrow} 0.\end{equation}
Finally, recall that $\limsup_{r \to \infty}\bar{\ell}_{k,\tilde{t}/2}^{\sss(r)}/r<1$ and choose any $\tilde{p}_{k,\tilde{t}} \in (\limsup_{r \to \infty}\bar{\ell}_{k,\tilde{t}/2}^{\sss(r)}/r,1)$. Then $\bar{\ell}_{k,\tilde{t}/2}^{\sss(r)}/r<\tilde{p}_{k,\tilde{t}}$ for all $r$ sufficiently large. In combination with \eqref{eq: EllVsBarEll}, this concludes the proof of Lemma \ref{lem: exclusiveFraction}.
\end{proof}
\begin{proof}[Proof of Lemma \ref{lem: lowerbound}]
Choose $\tilde{t},t, \tilde{p}_{k,\tilde{t}},p,\varepsilon,\alpha$ as in the statement of Lemma \ref{lem: lowerbound}. 
Further, conditionally on $d_\phi^-=r$, enumerate the $r$ children of $\phi$ by $j=1,2,\ldots,r$ in the order of their birth with the oldest child having index $1$. By Lemma \ref{lem: lateStop},
\begin{equation}\label{eq: RecaptauExcess}
\limsup_{r \to \infty}\mathbb{P}_k^{\mathcal{T}}\left( \dphigeqalphadir 
\geq pd^-_\phi \mid d_\phi^- =r\right) \leq \limsup_{r \to \infty}\mathbb{P}_k^{\mathcal{T}}\left( \dphigeqalphadir 
\geq pd^-_\phi \mid d_\phi^- =r, \ \tau = \sum_{j=1}^r \tilde{\tau}_k^{\sss(j)}+t\right)  
\end{equation}
Recall the random number $\ell_{k,\tilde{t}}^{\sss(r)}$ from \eqref{eq: ellDefinition} 
to bound \eqref{eq: RecaptauExcess} as
\begin{equation}\label{eq: IndicatorsAndTauExcess}
\begin{split}
&\limsup_{r \to \infty}\mathbb{P}_k^{\mathcal{T}}\left( \dphigeqalphadir 
\geq pd^-_\phi \mid d_\phi^- =r, \ \tau = \sum_{j=1}^r \tilde{\tau}_k^{\sss(j)}+t\right) \cr 
&=\limsup_{r \to \infty}\mathbb{P}_k^{\mathcal{T}}\left( \sum_{j=1}^r1\{d_j^- \geq \alpha\}
\geq pd^-_\phi \mid d_\phi^- =r, \ \tau = \sum_{j=1}^r \tilde{\tau}_k^{\sss(j)}+t\right)\cr 
&\leq \limsup_{r \to \infty}\mathbb{P}_k^{\mathcal{T}}\left( \sum_{j=\ell_{k,\tilde{t}}^{\sss(r)}}^r1\{d_j^- \geq \alpha\}
\geq pd^-_\phi-  \ell_{k,\tilde{t}}^{\sss(r)}\mid d_\phi^- =r, \ \tau = \sum_{j=1}^r \tilde{\tau}_k^{\sss(j)}+t\right).
\end{split}
\end{equation}
By definition of the tree $\mathcal{T}_k(\tau)$, for every $r \in \N$, the family $(d_j^-)_{j=1}^r$ is independent given the entire family $(\tilde{\tau}_k^{\sss(j)})_{j=1}^r$ of inter-birth waiting times, where $\tilde{\tau}_1$ is the absolute time of birth of the oldest child, and the time $\tau$ at which the birth process is stopped. By definition of $\ell_{k,\tilde{t}}^{\sss(r)}$, under the condition $\tau = \sum_{j=1}^r \tilde{\tau}_k^{\sss(j)}+t$, each of the vertices $\ell_{k,\tilde{t}}^{\sss(r)},\ell_{k,\tilde{t}}^{\sss(r)}+1, \ldots,r$ has at most time $\tilde{t}+t$ to produce offspring and, by the definition of $\alpha$, each of the indicators in \eqref{eq: IndicatorsAndTauExcess} thus takes the value one with a probability smaller than or equal to $p-\tilde{p}-\varepsilon$. Hence, the conditional probability obtained from further conditioning \eqref{eq: IndicatorsAndTauExcess} on $(\tilde{\tau}_k^{\sss(j)})_{j=1}^r$ can be bounded uniformly in the conditioning on $\tau$ and $(\tilde{\tau}_k^{\sss(j)})_{j=1}^r$ as
\begin{equation}\label{eq: stochastic bound}
\begin{split}
&\mathbb{P}_k^{\mathcal{T}}\left( \sum_{j=\ell_{k,\tilde{t}}^{\sss(r)}}^r1\{d_v^- \geq \alpha\}
\geq pd^-_\phi-  \ell_{k,\tilde{t}}^{\sss(r)}\mid d_\phi^- =r, \ \tau = \sum_{j=1}^r \tilde{\tau}_k^{\sss(j)}+t, (\tilde{\tau}_k^{\sss(j)})_{j=1,2,\ldots,r}\right)\cr
&\leq \mathbb{P}_k^{\mathcal{T}}\left(\sum_{j=\ell_{k,\tilde{t}}^{\sss(r)}}^{r}Z_j \geq pd^-_\phi-  \ell_{k,\tilde{t}}^{\sss(r)}\mid d_\phi^- =r\right), 
\end{split}
\end{equation}
where $(Z_j)_{j \in \N}$ is an i.i.d. family of $\mathrm{Bernoulli}(p-\tilde{p}_{k,\tilde{t}}-\varepsilon)$-random variables that is independent of everything else.
Note that the upper bound in \eqref{eq: stochastic bound} also is an upper bound on the respective (regarding the same $r$) probability under the limsup in \eqref{eq: IndicatorsAndTauExcess}, which formally follows from integrating out \eqref{eq: stochastic bound} with respect to the conditional distribution of the random vector $(\tilde{\tau}_j)_{j=1}^r$ given the event $\tau = \sum_{j=1}^r \tilde{\tau}_k^{\sss(j)}+t$. 
It remains to show that \eqref{eq: stochastic bound} vanishes as $r \to \infty$. First, we bound \eqref{eq: stochastic bound} as
\begin{equation}
\mathbb{P}_k^{\mathcal{T}}\left(\sum_{j=\ell_{k,\tilde{t}}^{\sss(r)}}^{r}Z_j \geq pd^-_\phi-  \ell_{k,\tilde{t}}^{\sss(r)}\mid d_\phi^- =r\right) \leq \mathbb{P}_k^{\mathcal{T}}\left(\frac{1}{r}\sum_{j=1}^{r}Z_j \geq p-  \frac{\ell_{k,\tilde{t}}^{\sss(r)}}{r}\mid d_\phi^- =r\right).   
\end{equation}
Next, we decompose
\begin{equation}\label{eq: ellDecomposition}
\begin{split}
\mathbb{P}_k^{\mathcal{T}}\left(\frac{1}{r}\sum_{j=1}^{r}Z_j \geq p-  \frac{\ell_{k,\tilde{t}}^{\sss(r)}}{r}\mid d_\phi^- =r\right)
&=\mathbb{P}_k^{\mathcal{T}}\left(\frac{1}{r}\sum_{j=1}^{r}Z_j \geq p-  \frac{\ell_{k,\tilde{t}}^{\sss(r)}}{r}, \ \ell_{k,\tilde{t}}^{\sss(r)} \leq \tilde{p}_{k,\tilde{t}}+\frac{\varepsilon}{2} \mid d_\phi^- =r\right)\cr & +\mathbb{P}_k^{\mathcal{T}}\left(\frac{1}{r}\sum_{j=1}^{r}Z_j \geq p-  \frac{\ell_{k,\tilde{t}}^{\sss(r)}}{r}, \ \ell_{k,\tilde{t}}^{\sss(r)} > \tilde{p}_{k,\tilde{t}}+\frac{\varepsilon}{2} \mid d_\phi^- =r\right).    
\end{split}
\end{equation}
By the choice of $\tilde{p}_{k,\tilde{t}}$, the second term in \eqref{eq: ellDecomposition} vanishes as $r \to \infty$. For the first term in \eqref{eq: ellDecomposition}, we note that the event $\ell_{k,\tilde{t}}^{\sss(r)} \leq \tilde{p}_{k,\tilde{t}}+\frac{\varepsilon}{2}$ implies $p-\ell_{k,\tilde{t}}^{\sss(r)}/r \geq p-\tilde{p}_{k,\tilde{t}}-\frac{\varepsilon}{2}>\mathbb{E}[Z_1].$ Hence, by the law of large numbers, the first term in \eqref{eq: ellDecomposition} also vanishes as $r \to \infty$.
This concludes the proof of Lemma \ref{lem: lowerbound}.
\end{proof}
\subsection{Proof of Proposition \ref{pro: MinorityFraction}}
\begin{proof}[Proof of Proposition \ref{pro: MinorityFraction}]
 We first decompose
\begin{equation}\label{eq: MinoritityFractionDecomposition}
\mathbb{E}\Big[\frac{d_{\phi_n}^\mathrm{mino}}{d_{\phi_n}} \mid G_n\Big]=\sum_{k=1}^2\left(\mathbb{E}\Big[\frac{d_{\phi_n}^{-,\mathrm{mino},(n)}}{d_{\phi_n}}\boldsymbol{1}\{X_{\phi_n}=k\} \mid G_n\Big]+\mathbb{E}\Big[\frac{d_{\phi_n}^{+,\mathrm{mino},(n)}}{d_{\phi_n}}\boldsymbol{1}\{X_{\phi_n}=k\} \mid G_n\Big]\right),    
\end{equation}
where, for any $n \in \N$ and vertex $w \in G_n$, \begin{equation}d_w^{-,\mathrm{mino},(n)}:=\#\{v \in V(G_n) \mid  X_v=1,\ v \to w \}
\end{equation}is the number of in-edges that $w$ receives from minority vertices in $G_n$,
and 
\begin{equation}
d_w^{+,\mathrm{mino},(n)}:=\#\{v \in V(G_n) \mid X_v=1, \ w \to v \}
\end{equation}
is the number of minority vertices in $G_n$ that receive one of the $m_{X_w}$ out-edges of $w$.
In the following paragraphs, we will discuss the asymptotic behavior of each of the two conditional expectations in  \eqref{eq: MinoritityFractionDecomposition} separately.
\paragraph{Controlling the in-edges from minority vertices.}
For $k=1,2$, from the known directed local limit in Theorem \ref{thm: locallimitdirectedmtypepam}, and since $\frac{d_{\phi_n}^{-,\mathrm{mino},(n)}}{d_{\phi_n}}\boldsymbol{1}\{X_{\phi_n}=k\}$ is measurable with respect to the incoming neighborhood of finite radius one of $\phi_n$ and its color, we obtain
\begin{equation}
\label{in-mino-limit}
\mathbb{E}\Big[\frac{d_{\phi_n}^{-,\mathrm{mino},(n)}}{d_{\phi_n}}\boldsymbol{1}\{X_{\phi_n}=k\} \mid G_n\Big] \stackrel{n \rightarrow \infty}{\rightarrow} \mu_k\mathbb{E}^{\mathcal{T}}_k\Big[\frac{d_{\phi}^{-,\mathrm{mino}}}{d_{\phi}^-+m_k}\Big] \qquad \text{almost surely},
\end{equation}
where 
\[d_{\phi}^{-,\mathrm{mino}}:=\#\{v \in \mathcal{T}_X(\tau) \mid  X_v=1, \ v \to \phi\} \subseteq \mathcal{T}_X(\tau),\] 
with $\mathcal{T}_X(\tau)$ denoting the rooted branching process tree defined in Theorem \ref{thm: locallimitdirectedmtypepam}.
Furthermore, 
\begin{equation}
\mathbb{E}^\mathcal{T}_k[~\cdot~]:=\mathbb{E}^{\mathcal{T}_X(\tau)}[~\cdot \mid X_\phi=k] 
\end{equation} 
denotes the expectation with respect to the distribution of $\mathcal{T}_X(\tau)$ conditioned on $X_\phi=k$, i.e., the root $\phi$ is conditioned to be a minority vertex for $k=1$, and a majority vertex for $k=2$.
Recall that $\mathcal{T}_X(\tau)$ is a multi-type branching process stopped at $\tau \sim \mathrm{Exp}(2)$ and with a non-constant growth rate. In particular, conditionally on $X_\phi=k$, the root $\phi$ produces offspring independently at times following the law of a Markovian pure birth process $\xi_{k}(\cdot)$ with birth rate of color-$l$-offspring at time $t$ given by $(\xi_{k}(t)+m_{k})\varphi_{k,l}$. 
The following lemma computes the right-hand side of \eqref{in-mino-limit}:
\begin{lemma}\label{lem: binomial}
Conditionally on $X_\phi=k$ and $d_\phi^-=d$, where $d \in \N$, $d_\phi^{-,\mathrm{mino}}$ is binomially distributed with parameters $d$ and $\varphi^{m}_{k,1}/(\varphi^{m}_{k,1}+\varphi^{m}_{k,2})$. 
In particular,
\[\mathbb{E}^{\mathcal{T}}_k\Big[\frac{d_{\phi}^{-,\mathrm{mino}}}{d_{\phi}^-+m_k}\Big]=\frac{\varphi^{m}_{k,1}}{\varphi^{m}_{k,1}+\varphi^{m}_{k,2}}\Big(1-m_k\mathbb{E}_k^\mathcal{T}\Big[\frac{1}{d_\phi^-+m_k}\Big]\Big).\]
\end{lemma}
\begin{proof}[Proof of Lemma \ref{lem: binomial}]
We will prove a slightly stronger statement by induction on $d$, the induction statement being that, for every $s \in \N_0$ and $k \in \{1,2\}$  conditionally on $X_\phi=k$ and $d_\phi^-=d+s$, the number of minority children among the first $d$ children of $\phi$ is binomially distributed with parameters $d$ and $\varphi^{m}_{k,1}/(\varphi^{m}_{k,1}+\varphi^{m}_{k,2})$.

To get started, let $t \geq 0$ be given, and let $\tau_{t,\mathrm{mino}}$ denote the waiting time after $t$ until the birth of the first minority-child of $\phi$ born after time $t$. Similarly, let $\tau_{t,\mathrm{majo}}$ denote the waiting time after $t$ until the birth of the first majority-child of $\phi$ born after time $t$.
Then 
\begin{equation}
\tau_{t,\mathrm{mino}} \sim \mathrm{Exp}\left((m_k+\xi_k(t))\varphi^{m}_{k,1}\right), \qquad 
\tau_{t,\mathrm{majo}} \sim \mathrm{Exp}\left((m_k+\xi_k(t))\varphi^{m}_{k,2}\right).
\end{equation}
In particular,
\begin{equation}
\mathbb{P}_k^{\mathcal{T}}(\tau_{t,\mathrm{mino}} < \tau_{t,\mathrm{majo}})=\frac{(m_k+\xi_k(t))\varphi^{m}_{k,1}}{(m_k+\xi_k(t))\varphi^{m}_{k,1}+(m_k+\xi_k(t))\varphi^{m}_{k,2}}=\frac{\varphi^{m}_{k,1}}{\varphi^{m}_{k,1}+\varphi^{m}_{k,2}} 
\end{equation}
independently of $t$.
In particular, for $t=0$, this proves the initialization of the  induction hypothesis for $d=1$. 

Next, assume that the induction hypothesis holds for some $d$. Consider the event $d^-_\phi =d+1+s$ for any $s \in \N_0$, and let $\tau_{d}$ denote the time of birth of child $d$. Further, let $d_{\phi,d}^{-,\mathrm{mino}}$ denote the number of minority children among the first (at most) $d$ children of $\phi$. Then,
\begin{equation}
\begin{split}
&\mathbb{P}^{\mathcal{T}}_k(d_{\phi,d+1}^{-,\mathrm{mino}}=d_{\phi,d}^{-,\mathrm{mino}}+1 \mid d_\phi^-=d+1+s)\cr&=\int_0^\infty \mathbb{P}^{\mathcal{T}}_k(d_{\phi,d+1}^{-,\mathrm{mino}}=d_{\phi,d}^{-,\mathrm{mino}}+1 \mid \tau_{d}=t)\mathbb{P}_k^{\tau_{d}}(dt)\cr
&=\int_0^\infty \mathbb{P}^{\mathcal{T}}_k(\tau_{t,\mathrm{mino}} < \tau_{t,\mathrm{majo}})\mathbb{P}_k^{\tau_{d}}(dt)=\frac{\varphi^{m}_{k,1}}{\varphi^{m}_{k,1}+\varphi^{m}_{k,2}}.
\end{split}    
\end{equation}
This proves the induction step. With the choice $s=0$ this thus concludes the proof of the first statement of Lemma \ref{lem: binomial}. 

For the second statement, we condition on the value of $d_\phi^-+m_k$ to obtain
\begin{equation}\label{eq: MinoInFractionTotalExp}
\mathbb{E}^{\mathcal{T}}_k\Big[\frac{d_{\phi}^{-,\mathrm{mino}}}{d_{\phi}^-+m_k}\Big]=\sum_{r=m_k}^\infty\mathbb{E}_k^{\mathcal{T}}\Big[\frac{d_{\phi}^{-,\mathrm{mino}}}{r} \mid d_\phi^{-}+m_k=r\Big]\mathbb{P}_k^\mathcal{T}(d_\phi^-+m_k=r).
\end{equation}
By the first statement of the lemma, conditionally on $d_\phi^-+m_k=r$ and $X_\phi=k$, the random variable $d_{\phi}^{-,\mathrm{mino}}$ is binomially distributed with parameters $r-m_k$ and $\varphi^{m}_{k,1}/(\varphi^{m}_{k,1}+\varphi^{m}_{k,2})$. Thus,  
\begin{equation}\label{eq: r cancellation}
\mathbb{E}_k^{\mathcal{T}}\Big[\frac{d_{\phi}^{-,\mathrm{mino}}}{r} \mid d_\phi^{-}+m_k=r\Big]=\frac{\varphi^{m}_{k,1}}{\varphi^{m}_{k,1}+\varphi^{m}_{k,2}}\left(1-\frac{m_k}{r}\right).  
\end{equation}
Inserting \eqref{eq: r cancellation} into \eqref{eq: MinoInFractionTotalExp} finishes the proof of Lemma \ref{lem: binomial}.
\end{proof}
\paragraph{The degrees and minority-out degrees are asymptotically independent.}
The behavior of $d_{\phi_n}^{+,\mathrm{mino}}/d_{\phi_n}$ is not described  by the directed local limit in Theorem \ref{thm: locallimitdirectedmtypepam}, which only provides information on finite incoming neighborhoods of $\phi_n$ and its color. In particular, controlling the correlation between $d_{\phi_n}^{+,\mathrm{mino}}$ and $d_{\phi_n}$ needs an additional argument. In Lemma \ref{lem: asymptotic uncorrelated} below we will describe the asymptotic joint law of $(d_{\phi_n},d_{\phi_n}^{+,\mathrm{mino},(n)},X_{\phi_n})$. The proof of Lemma \ref{lem: asymptotic uncorrelated} is an extension of the proof of \cite[Theorem 2.2]{Jo13}, which describes the distribution of $d_{\phi_n}$ as $n \to \infty$, and also serves as the basis of the proof of Theorem \ref{thm: locallimitdirectedmtypepam}.
It is based on the following stochastic recursion:
\begin{lemma}[Lemma 3.3 in \cite{Jo13}]\label{lem: JordansLemma}
For $n \in \N$, let $A_n$ and $B_n$ be random variables taking non-negative values, $\xi_n,R_n$ random variables taking real values, and $a$ and $b$ positive constants such that
\[B_{n+1}-B_n=\frac{1}{n}(A_n-bB_n+\xi_n)+R_{n+1},\]
where 
\begin{enumerate}
    \item $A_n \stackrel{n \to \infty}{\rightarrow} a$ almost surely;
    \item $\sum_{n=1}^\infty R_n<\infty;
    $
    \item $\mathbb{E}[\xi_n]=0$ and $(\xi_n)_{n \in \N}$ is bounded. 
\end{enumerate}
Then $B_n \stackrel{n \to \infty}{\rightarrow} a/b$ almost surely.
\end{lemma}

We start our extension by recalling and introducing some notation.
Recall that $M=\mu_1 m_1+\mu_2m_2$ and, for every $y \in (0,M)$, define 
\begin{equation}\label{eq: binomP}
 p_{y,k}=\frac{\kappa(1,k)y}{\kappa(1,k)y+\kappa(2,k)( M-y)}.    
\end{equation}
Then, the statement of Lemma \ref{lem: asymptotic uncorrelated} is as follows:
\begin{lemma}
\label{lem: asymptotic uncorrelated}
For any $k \in \{1,2\}$, any $\omega \in \{0,1,\ldots,m_k\}$ and any $d \in \{m_k,m_k+1,\ldots\}$,
as $n \to \infty$, the sequence of random variables 
\begin{equation}
\left(\mathbb{P}(d_{\phi_n}=d,d_{\phi_n}^{+,\mathrm{mino},(n)}=\omega,X_{\phi_n}=k \mid G_n)\right)_{n \in \N}     
\end{equation}
converges almost surely to
\begin{equation}\label{eq: LimitingJointDistribution}
\mu_k \cdot \frac{2\Gamma(m_k+\frac{2}{\varphi_k})}{\varphi_k \Gamma(m_k)}\frac{\Gamma(d)}{\Gamma(d+\frac{2}{\varphi_k}+1)}\cdot \binom{m_k}{\omega}p^\omega_{\eta_k^{\boldsymbol{m}},k}(1-p_{\eta_k^{\boldsymbol{m}},k})^{m_k-\omega}.    
\end{equation}
\end{lemma}
\begin{proof}[Proof of Lemma \ref{lem: asymptotic uncorrelated}]
The proof is an adaption of the proof of \cite[Theorem 2.2]{Jo13} to cover the additional constraint on $d_{\phi_n}^{+,\mathrm{mino},(n)}$. First define 
\begin{equation}
\bar{\mathcal{N}}_{k,d,\omega}^{\sss(n)}:=\{v \in V(G_n) \mid d_v=d,d_v^{+,\mathrm{mino},(n)}=\omega,X_v=k\}
\end{equation}
and 
\begin{equation}
Z_{k,d,\omega}^{\sss(n)}:=\frac{1}{n+n_0} \#\bar{\mathcal{N}}_{k,d,\omega}^{\sss(n)}.    
\end{equation}
We want to prove almost sure convergence of $(Z_{k,d,\omega}^{\sss(n)})_{n \in \N}$ to the quantity in \eqref{eq: LimitingJointDistribution}.
First note that, for each $u \in G_n$ with $d_u=d$ and $X_u=k$, by the attachment rule \eqref{eq: AttachmentEmpirical} and Proposition \ref{pro: Jordan},
\begin{equation}\label{eq: binomAttZero}
\begin{split}
&\mathbb{P}(v_{n+1} \text{ does not connect to } u \mid G_n)\cr 
&=\sum_{l=1}^2\mu_l\left(1-\frac{d\kappa(k,l)}{2(n+n_0)\sum_{\tilde{l}=1}^2\kappa(\tilde{l},l)\tilde{Y}_{\tilde{l}}^{\boldsymbol{m},(n)}} \right)^{m_l}\cr
&=\sum_{l=1}^2\mu_l\left(1-m_l\frac{d\kappa(k,l)}{2(n+n_0)\sum_{\tilde{l}=1}^2\kappa(\tilde{l},l)\tilde{Y}_{\tilde{l}}^{\boldsymbol{m},(n)}}\right)+o_{\mathbb{P}}(1/n)\cr 
&=1-\frac{d}{2(n+n_0)}\varphi_{k}^{\boldsymbol{m}}+o_{\mathbb{P}}(1/n).
\end{split}
\end{equation}
Here, the second equality follows from a first-order Taylor expansion in combination with the assumption of strict positivity of $\kappa$.
Furthermore, for a sequence of random variables $(\xi_n)_{n \in \N}$ and a sequence $(a_n)_{n \in \N}$ of positive numbers, we write $\xi_n=o_\mathbb{P}(a_n)$ when $(\xi_n/a_n)_{n \in \N}$ converges in probability to zero.
Similarly, for each such $u$, \begin{equation}\label{eq: binomAttOne}
\begin{split}
&\mathbb{P}(v_{n+1} \text{ connects to } u \text{ via a single edge}\mid G_n)
=\frac{d}{2(n+n_0)}\varphi_{k}^{\boldsymbol{m}}+o_{\mathbb{P}}(1/n).
\end{split}
\end{equation}
From \eqref{eq: binomAttZero} and \eqref{eq: binomAttOne}, we obtain, for every vertex $u \in G_n$ with prescribed degree and color, that
\begin{equation}\label{eq: multi-edges}
\mathbb{P}(v_{n+1} \text{ and } u \text{ are connected by a multi-edge} \mid G_n)=o_\mathbb{P}(1/n).    
\end{equation}
Next, \eqref{eq: binomAttZero}, \eqref{eq: binomAttOne}, \eqref{eq: multi-edges}, and linearity of the conditional expectation, give
\begin{align}\label{eq: JordansRecursion}
\mathbb{E}[Z_{k,d,\omega}^{\sss(n+1)}\mid G_n]&=\frac{1}{n+n_0+1}\Big((n+n_0)Z_{k,d,\omega}^{\sss(n)}(1-\frac{d}{2(n+n_0)}\varphi_{k}^{\boldsymbol{m}}+o_{\mathbb{P}}(1/n))\\
&+Z_{k,d-1,\omega}^{\sss(n)}(\frac{d-1}{2}\varphi_{k}^{\boldsymbol{m}}+o_{\mathbb{P}}(1))+\mu_k\delta_{m_k,d}\binom{m_k}{\omega}p^\omega_{\eta_k^{\boldsymbol{m}},k}(1-p_{\eta_k^{\boldsymbol{m}},k})^{m_k-\omega}+o_{\mathbb{P}}(1) \Big).\nonumber    
\end{align}
Here, in view of the independent attachment rule, the last term in \eqref{eq: JordansRecursion} addresses the case that $v_{n+1} \in \mathcal{N}_{k,d,\omega}^{\sss(n+1)}$, where $\delta_{m_k,d}=1$ if $m_k=d$ and $\delta_{m_k,d}=0$, else.
We may reformulate \eqref{eq: JordansRecursion} as
\begin{equation}\label{eq: JordansRecursion2}
\begin{split}
&Z_{k,d,\omega}^{\sss(n+1)}-Z_{k,d,\omega}^{\sss(n)}=\frac{1}{n+n_0+1}\Big(-Z_{k,d,\omega}^{\sss(n)}(1+\frac{d}{2}\varphi_{k}^{\boldsymbol{m}}+o_{\mathbb{P}}(1)) +Z_{k,d-1,\omega}^{\sss(n)}(\frac{d-1}{2}\varphi_{k}^{\boldsymbol{m}}+o_{\mathbb{P}}(1))\cr 
&+\mu_k\delta_{m_k,d}\binom{m_k}{\omega}p^\omega_{\eta_k^{\boldsymbol{m}},k}(1-p_{\eta_k^{\boldsymbol{m}},k})^{m_k-\omega}+o_{\mathbb{P}}(1)+ (n+n_0+1)(Z_{k,d,\omega}^{\sss(n+1)}-\mathbb{E}[Z_{k,d,\omega}^{\sss(n+1)} \mid G_n])\Big).  
\end{split}    
\end{equation}
Note that the noise term $(n+n_0-1)(Z_{k,d,\omega}^{\sss(n+1)}-\mathbb{E}[Z_{k,d,\omega}^{\sss(n+1)} \mid G_n])$ in \eqref{eq: JordansRecursion2} satisfies
\begin{align}
(n+n_0+1)(Z_{k,d,\omega}^{\sss(n+1)}-\mathbb{E}[Z_{k,d,\omega}^{\sss(n+1)} \mid G_n]) &\leq \vert (n+n_0+1)Z_{k,d,\omega}^{\sss(n+1)}-(n+n_0)Z_{k,d,\omega}^{\sss(n+1)}\vert \nonumber\\
&\qquad+\mathbb{E}[\vert (n+n_0+1)Z_{k,d,\omega}^{\sss(n+1)}- (n+n_0+1)Z_{k,d,\omega}^{\sss(n)} \vert \mid G_n] \nonumber\\
&\leq 2\max\{m_1,m_2\}.
\end{align}
We proceed to prove almost sure convergence of $(Z_{k,d,\omega}^{\sss(n)})_{n \in \N}$ by induction on $d \geq m_k$.
In the initial case $d=m_k$, $Z_{k,d-1,\omega}^{\sss(n)}=0$ almost surely, and we  apply Lemma \ref{lem: JordansLemma} with $B_n:=Z_{k,m_k,\omega}^{\sss(n)}$, $A_n:=\mu_k\binom{m_k}{\omega}p^\omega_{\eta_k^{\boldsymbol{m}},k}(1-p_{\eta_k^{\boldsymbol{m}},k})^{m_k-\omega}+o_{\mathbb{P}}(1)$, $b=1+\frac{m_k}{2}\varphi_k^{\boldsymbol{m}}$, $\xi_n:=(n+n_0+1)(Z_{k,d,\omega}^{\sss(n+1)}-\mathbb{E}[Z_{k,d,\omega}^{\sss(n+1)} \mid G_n])$ and $R_n=0$, to obtain
\begin{equation}
Z_{k,m_k,\omega}^{\sss(n)} \stackrel{n \to \infty}{\rightarrow}  \mu_k\frac{2}{\varphi_k^{\boldsymbol{m}}m_k+2}\binom{m_k}{\omega}p^\omega_{\eta_k^{\boldsymbol{m}},k}(1-p_{\eta_k^{\boldsymbol{m}},k})^{m_k-\omega},
\end{equation}
as required, and this initializes the induction hypothesis.
Next, consider $d>m_k$ and assume that the statement of Lemma \ref{lem: asymptotic uncorrelated} holds for $d-1$.
Then we apply Lemma \ref{lem: JordansLemma} to \eqref{eq: JordansRecursion2} with $B_n:=Z_{k,d,\omega}^{\sss(n)}$, $A_n:=Z_{k,d-1,\omega}^{\sss(n)}(\frac{d-1}{2}\varphi_k^{\sss(m)}+o_{\mathbb{P}(1)})$, $b=1+\frac{d}{2}\varphi_k^{\boldsymbol{m}}$, $\xi_n:=(n+n_0+1)(Z_{k,d,\omega}^{\sss(n+1)}-\mathbb{E}[Z_{k,d,\omega}^{\sss(n+1)} \mid G_n])$ and $R_n=0$, to obtain
\begin{align}
Z_{k,d,\omega}^{\sss(n)} &\stackrel{n \to \infty}{\rightarrow}\frac{1}{1+\frac{d}{2}\varphi_k^{\boldsymbol{m}}}\frac{d-1}{2}\varphi_k^{\boldsymbol{m}}\mu_k \ \frac{2\Gamma(m_k+\frac{2}{\varphi_k^{\boldsymbol{m}}})}{\varphi_k^{\boldsymbol{m}} \Gamma(m_k)}\frac{\Gamma(d-1)}{\Gamma(d+\frac{2}{\varphi_k^{\boldsymbol{m}}})} \binom{m_k}{\omega}p^\omega_{\eta_k^{\boldsymbol{m}},k}(1-p_{\eta_k^{\boldsymbol{m}},k})^{m_k-\omega}\nonumber\\
&=\mu_k \frac{2\Gamma(m_k+\frac{2}{\varphi_k^{\boldsymbol{m}}})}{\varphi_k^{\boldsymbol{m}} \Gamma(m_k)}\frac{\Gamma(d)}{\Gamma(d+\frac{2}{\varphi_k^{\boldsymbol{m}}}+1)} \binom{m_k}{\omega}p^\omega_{\eta_k^{\boldsymbol{m}},k}(1-p_{\eta_k^{\boldsymbol{m}},k})^{m_k-\omega}.
\end{align}
This advances the induction hypothesis, and thus completes the proof of Lemma \ref{lem: asymptotic uncorrelated}.
\end{proof}

\vspace{-0.5cm}
\paragraph{The expected inverse degree of the root.}
\begin{lemma}\label{lem: Exp1overDeg}
Let $\psi_1(z):=\frac{d^2}{dz^2}\log(\Gamma(z))$ be the trigamma function. Then 
\begin{equation}
\mathbb{E}^\mathcal{T}_k\Big[\frac{1}{d_\phi^{-}+m_k}\Big]=\frac{2\Gamma(m_k+\frac{2}{\varphi_k^{\boldsymbol{m}}})}{\varphi_k^{\boldsymbol{m}} \Gamma(m_k)\Gamma(1+\frac{2}{\varphi_k^{\boldsymbol{m}}})}\left(\psi_1(1+\frac{2}{\varphi_k^{\boldsymbol{m}}})-\sum_{r=1}^{m_k-1}\frac{\beta(r,1+\frac{2}{\varphi_k^{\boldsymbol{m}}})}{r}\right).
\end{equation}
\end{lemma}
\begin{proof}[Proof of Lemma \ref{lem: Exp1overDeg}]
From the known distribution of $d_\phi$ (see \cite[Theorem 4.4]{AnBaBhPi26}),
we obtain
\begin{equation}
\begin{split}
\mathbb{E}^\mathcal{T}_k\Big[\frac{1}{d_\phi^-+m_k}\Big]&=\sum_{r=m_k}^\infty \frac{1}{r}\mathbb{P}_k^\mathcal{T}(d_\phi^-+m_k=r)=\frac{2\Gamma(m_k+\frac{2}{\varphi_k^{\boldsymbol{m}}})}{\varphi_k^{\boldsymbol{m}} \Gamma(m_k)}\sum_{r=m_k}^\infty \frac{\Gamma(r)}{r\Gamma(r+\frac{2}{\varphi_k^{\boldsymbol{m}}}+1)}
\end{split}    
\end{equation}
Now recall the definition of the $\beta$-function:\[\beta(z_1,z_2)=\frac{\Gamma(z_1)\Gamma(z_2)}{\Gamma(z_1+z_2)}=\int_0^1 t^{z_1-1}(1-t)^{z_2-1} \text{d}t\text{ for all } z_1,z_2 \in \mathbb{C} \text{ with } \text{Re}(z_1)>0, \text{Re}(z_2)>0.\]
Thus, we have for any $r\in \N$ and any $a>2$:
\begin{equation}
\begin{split}
\frac{\Gamma(r)
}{r\Gamma(r+a)}=\frac{\beta(r,a)}{r\Gamma(a)}=\frac{1}{r\Gamma(a)}\int_0^1 t^{r-1}(1-t)^{a-1} \, \text{d}t.
\end{split}    
\end{equation}
In combination with monotone convergence, this allows us to rewrite 
\begin{equation}
\begin{split}
\sum_{r=m_k}^\infty \frac{\Gamma(r)}{r\Gamma(r+a)}&=\frac{1}{\Gamma(a)}\int_0^1\frac{(1-t)^{a-1}}{t}\sum_{r=m_k}^\infty \frac{t^r}{r} \ \text{d}t\cr 
&=\frac{1}{\Gamma(a)}\int_0^1\frac{(1-t)^{a-1}}{t}\left(-\log(1-t)-\sum_{r=1}^{m_k-1} \frac{t^r}{r}\right) \ \text{d}t\cr
&=\frac{1}{\Gamma(a)}\int_0^1\frac{z^{a-1}}{1-z}\left(-\log z-\sum_{r=1}^{m_k-1} \frac{(1-z)^r}{r}\right) \ \text{d}z\cr
&=\frac{\psi_1(a)}{\Gamma(a)}-\frac{1}{\Gamma(a)}\sum_{r=1}^{m_k-1}\frac{\beta(a,r)}{r}.
\end{split}    
\end{equation} 
Finally, inserting $a=\lambda_k=\frac{2}{\varphi_k^{\boldsymbol{m}}}+1$ (cf. Corollary \ref{cor: exponents}) completes the proof of Lemma \ref{lem: Exp1overDeg}.
\end{proof}
\vspace{-1cm}

\paragraph{Finalization of the proof of Proposition \ref{pro: MinorityFraction}.}
Recalling that $\mathbb{P}_k^{\mathcal{T}}(d_\phi^-+m_k)=\frac{2\Gamma(m_k+\frac{2}{\varphi_k^{\boldsymbol{m}}})}{\varphi_k^{\boldsymbol{m}} \Gamma(m_k)}\frac{\Gamma(d)}{\Gamma(d+\frac{2}{\varphi_k^{\boldsymbol{m}}}+1)}$, by Lemma \ref{lem: asymptotic uncorrelated},
\begin{equation}\label{eq: CondExpFractionConv}
\mathbb{E}\Big[\frac{d_{\phi_n}^{+,\mathrm{mino},(n)}}{d_{\phi_n}}\boldsymbol{1}\{X_{\phi_n}=k\} \mid G_n\Big] \stackrel{n \to \infty}{\rightarrow} \mu_k\mathbb{E}_k^\mathcal{T}\Big[\frac{1}{d_\phi^-+m_k}\Big]m_kp_{\eta_1^{\boldsymbol{m}},k} \quad \text{almost surely,}  
\end{equation}
where, by \eqref{eq: binomP} and \eqref{eq: varphik}, $m_kp_{\eta_1^{\boldsymbol{m}},k}=\frac{1}{\mu_k}\eta_k^{\boldsymbol{m}}\varphi_k^{\boldsymbol{m}}.$
To formally reason \eqref{eq: CondExpFractionConv}, note that by Lemma \ref{lem: asymptotic uncorrelated} and Scheff\'e's Lemma (e.g., \cite[Lemma 5.10]{Wi91})
\begin{align}
\sum_{k=1}^2\sum_{d=m_k}^\infty \sum_{\omega=1}^{m_k} \Big\vert &\mathbb{P}(d_{\phi_n}=d,d_{\phi_n}^{+,\mathrm{mino},(n)}=\omega,X_{\phi_n}=k \mid G_n)\\
&\quad -\mu_k \cdot \frac{2\Gamma(m_k+\frac{2}{\varphi_k})}{\varphi_k \Gamma(m_k)}\frac{\Gamma(d)}{\Gamma(d+\frac{2}{\varphi_k}+1)}\cdot \binom{m_k}{\omega}p^\omega_{\eta_k^{\boldsymbol{m}},k}(1-p_{\eta_k^{\boldsymbol{m}},k})^{m_k-\omega} \Big\vert \stackrel{n \to \infty}{\rightarrow} 0,\nonumber
\end{align}
almost surely, which, by boundedness of the fraction $d_{\phi_n}^{+,\mathrm{mino},(n)}/d_{\phi_n},$ implies \eqref{eq: CondExpFractionConv}.
Finally, we set $A_{k,\boldsymbol{m},\boldsymbol{\varphi}}:=\mathbb{E}^\mathcal{T}_k[\frac{1}{d_\phi^{-}+m_k}]$, and insert the expression from Lemma \ref{lem: Exp1overDeg}. In combination with Lemma \ref{lem: binomial}, we thus obtain a description of the almost sure limits of all for terms on the decomposition \eqref{eq: MinoritityFractionDecomposition}, which concludes the proof of Proposition \ref{pro: MinorityFraction}.
\end{proof}
\section{Discussion}\label{sec: discussion}
\paragraph{Undirected local limit for the colored PAM.}
In this paper we studied the PageRank in the undirected colored model; yet, we did not derive its local limit. Instead we based our analysis solely on the known local limit for the directed version. Then, a natural open problem is to derive the local limit for the undirected model with colors.
The approach of Antunes et al.\ \cite{AnBaBhPi26} for obtaining the directed local limit by reducing the problem to discrete dynamical systems of the same type that had already been studied by Jordan \cite{Jo13} (cf.\ \eqref{eq: LyapunovFunction} and Lemma \ref{lem: JordansLemma} above) is based on {\em fringe convergence}, which works particularly well for trees. Its structural strength lies in the fact that it provides local convergence almost surely, going beyond local convergence in probability, which often follows from a second-moment method.  
Recall from \cite{BeBoCh14} that the known local limit for the undirected PAM without colors is a multi-type branching process tree, where vertices carry a two-dimensional type in $[0,1] \times \{y,o \}$ describing the absolute time of birth and the relative age (\textbf{y}ounger or \textbf{o}lder) compared to that of root in the underlying PAM.  
We may conjecture that a potential candidate for the local limit in the undirected PAM with colors is a branching process tree where vertex types also include information on the relative age.
It is an interesting question whether the proof of \cite[Theorem 4.3]{AnBaBhPi26} has a natural adaption to undirected graphs when splitting the test-graph, which under the above conjecture can be assumed to be a tree, into a subgraph consisting of older vertices and a subgraph consisting of younger vertices. A first step in this direction can be seen in the proof of Lemma \ref{lem: asymptotic uncorrelated}, where we combine a constraint on the incoming neighborhood (the in-degree) with a constraint on the outgoing neighborhood (number of out-edges that connect to minority vertices). 
Another approach to obtain an undirected local limit for the colored PAM is to obtain a P\'olya-urn representation as established for the model without colors by Berger et al.\  \cite{BeBoCh14}. Note that as in the case of the undirected PAM with i.i.d.\ random numbers of half-edges for the new vertices (cf.\  \cite{GaHaHoRa22} for a study of various sequential PAMs with i.i.d.\ out-degrees satisfying a moment condition), implementing colors to the model destroys exchangeability, and calculations thus cannot be simplified by employing de Finetti's Theorem.  
In particular, following the strategy of proof employed in \cite{GaHaHoRa22}, the steps to prove the undirected local convergence for the colored PAM include:
\begin{itemize}
    \item[$\rhd$] Reduce problem to the tree-case ($m=1$) by identifying the model for $m > 1$ with a suitable \textbf{collapsed} tree-model.
    \item[$\rhd$] Find a candidate for the P\'olya urn representation of the tree-model and verify by explicit calculation that it has the same distribution as the colored PAM at least on the level of neighborhoods around uniformly chosen vertices.
    \item[$\rhd$] Analyse the long-term behavior of the representation on the local level.
\end{itemize}

\paragraph{Precise asymptotics of PageRank in when undirected local limit is known.}
Note that we were able to prove the power-law bounds for PageRank in Theorem \ref{thm: Power-law bounds} without knowing the local limit (and even more, without knowing that this local limit exists at all) for the undirected PAM with colors. Knowing the local limit gives rise to a convergence statement for PageRank beyond power-law bounds in the following sense:  
For any locally finite, potentially infinite connected rooted graph $(G,\phi)$, we define
 \begin{equation}\label{eq: vonNeumann}
 R_\phi(G)=(1-c)\sum_{s=0}^\infty c^s \sum_{j \in V(G)} (\boldsymbol{P}^s)_{j\phi}. 
\end{equation}   
Here, $\boldsymbol{P}=(p_{ij})_{i,j \in \N}$ is the (possibly infinite-dimensional) matrix which is defined by $p_{ij}=a_{ij}/d_{v_i}$, with $\boldsymbol{A}=(a_{ij})_{i,j \in \N}$ denoting the adjacency matrix of $G$.
Then \cite[Theorem 9.2]{vH24} guarantees that for any sequence of directed random graphs that converges locally weakly in the marked backward sense with limit $(G_\infty,\phi)$, the sequence of PageRanks at a uniformly chosen vertex from the respective graph converges weakly to the limiting PageRank $R_\phi(G_\infty)$ as defined in \eqref{eq: vonNeumann}, where $\mathbb{E}[R_\phi(G_\infty)]<\infty$. 
This result also covers the case of undirected graphs that converge locally weakly.

\paragraph{Acknowledgments.} 
The work of RvdH and NL was supported by the
Netherlands Organisation for Scientific Research (NWO) through the Gravitation NETWORKS grant 024.002.003, and by the National Science Foudation under Grant No. DMS-1928930 while the authors were in residence at the Simons Laufer Mathematical Sciences Institute in Berkeley, California, during the Spring semester 2025.
\printbibliography
\end{document}